\documentclass[12pt, letterpaper]{amsart}

\usepackage[margin=1in]{geometry}
\usepackage{enumerate}
\usepackage{tikz}
\usetikzlibrary{cd}

\title{The radiation field on product cones}
\author[D. Baskin]{Dean Baskin}
\email{dbaskin@math.tamu.edu}
\address{Department of Mathematics, Texas A\&M University \\ Mailstop
  3368 \\ College Station, TX 77843}
\author[J.L. Marzuola]{Jeremy L. Marzuola}
\email{marzuola@math.unc.edu}
\address{Department of Mathematics, UNC-Chapel Hill \\ CB\#3250
  Phillips Hall \\ Chapel Hill, NC 27599}

\usepackage{xcolor}

\newcommand{\bl}{\mathrm{b}}
\newcommand{\e}{\mathrm{e}}
\newcommand{\us}{\widetilde{u}_\sigma}
\newcommand{\fs}{\widetilde{f}_\sigma}
\newcommand{\ut}{\widetilde{u}}
\newcommand{\rhot}{\tilde{\rho}}
\newcommand{\cE}{\mathcal{E}}
\newcommand{\cQ}{\mathcal{Q}}
\newcommand{\tG}{\widetilde{G}}

\newcommand{\Sbstar}{{}^{\bl}S^{*}}
\newcommand{\Sbdotstar}{{}^{\bl}\dot{S}^*}
\newcommand{\Sestar}{{}^\e S^*}
\newcommand{\Tb}{{}^\bl T}
\newcommand{\Tbstar}{{}^{\bl}T^*}
\newcommand{\Tbdot}{{}^{\bl}\dot{T}}
\newcommand{\Tbdotstar}{\Tbdot^*}
\newcommand{\Tbstarbar}{\overline{\Tbstar}}
\newcommand{\Te}{{}^{\e}T}
\newcommand{\Testar}{\Te^*}

\newcommand{\xib}{\underline{\xi}}
\newcommand{\zetab}{\underline{\zeta}}
\newcommand{\taub}{\underline{\tau}}
\newcommand{\taue}{\tau_\e}
\newcommand{\xie}{\xi_\e}
\newcommand{\zetae}{\zeta_\e}

\newcommand{\norm}[1]{{\left\lVert{#1}\right\rVert}}

\newcommand{\lap}{\Delta}
\newcommand{\grad}{\nabla}
\newcommand{\Ps}{\widetilde{P}_\sigma}
\newcommand{\Ph}{\widetilde{P}_h}
\DeclareMathOperator{\Id}{Id}
\DeclareMathOperator{\Op}{Op}
\newcommand{\Opbh}{\Op_{\bl,h}}

\newcommand{\CI}{C^\infty}
\newcommand{\dom}{\mathcal{D}}
\newcommand{\domh}{\dom_h}
\newcommand{\domst}{\widetilde{\dom}}
\newcommand{\Hb}{H_\bl}
\newcommand{\Hbd}{H_{\bl, \dom}}
\newcommand{\Hbdst}{H_{\bl, \domst}}
\newcommand{\cX}{\mathcal{X}}
\newcommand{\cY}{\mathcal{Y}}
\newcommand{\tow}{\mathrm{ftr}}
\newcommand{\away}{\mathrm{past}}

\newcommand{\Psib}{\Psi_{\bl}}
\newcommand{\Psibinf}{\Psi_{\bl,\infty}}
\newcommand{\Psibh}{\Psi_{\bl,h}}
\newcommand{\Psibhc}{\Psibh^{\mathrm{comp}}}
\DeclareMathOperator{\WF}{WF}
\newcommand{\WFb}{\WF_\bl}
\newcommand{\WFbh}{\WF_{\bl,h}}
\newcommand{\WFbdst}{\WF_{\bl,\domst}}
\newcommand{\WFbdprime}{\WF_{\bl,\domst'}}
\newcommand{\ellb}{\operatorname{ell}_\bl}
\newcommand{\ellbh}{\operatorname{ell}_{\bl,h}}
\newcommand{\sigmab}{\sigma_\bl}
\newcommand{\sigmabh}{\sigma_{\bl,h}}
\newcommand{\Sigmadot}{\dot{\Sigma}}
\newcommand{\Sigmah}{\Sigma_h}
\DeclareMathOperator{\Diff}{Diff}
\newcommand{\Diffb}{\Diff_\bl}
\newcommand{\Diffbh}{\Diff_{\bl,h}}
\newcommand{\Diffe}{\Diff_{\mathrm{e}}}
\newcommand{\module}{\mathcal{M}}
\newcommand{\pphg}{\mathcal{A}_{\mathrm{pphg}}}

\newcommand{\complexes}{\mathbb{C}}
\newcommand{\integers}{\mathbb{Z}}
\newcommand{\naturals}{\mathbb{N}}
\newcommand{\reals}{\mathbb{R}}
\newcommand{\sphere}{\mathbb{S}}
\renewcommand{\Im}{\operatorname{Im}}
\renewcommand{\Re}{\operatorname{Re}}

\DeclareMathOperator{\mf}{mf}
\DeclareMathOperator{\cf}{cf}
\newcommand{\scri}{\mathcal{I}}

\newcommand{\ds}{\displaystyle}
\newcommand{\abs}[1]{\lvert#1\rvert}
\newcommand{\ang}[1]{\left\langle #1 \right\rangle}

\newcommand{\pd}[1][]{\partial_{#1}}

\newcommand{\met}{k}
\newcommand{\bo}{\mathcal{O}}

\newtheorem{theorem}{Theorem}
\newtheorem{lemma}[theorem]{Lemma}
\newtheorem{proposition}[theorem]{Proposition}
\newtheorem{corollary}[theorem]{Corollary}

\theoremstyle{definition}
\newtheorem{definition}[theorem]{Definition}
\newtheorem{remark}[theorem]{Remark}

\numberwithin{theorem}{section}

\begin{document}


\begin{abstract}
  We consider the wave equation on a product cone and find a joint
  asymptotic expansion for solutions near null and future infinities.
  The rates of decay seen in the expansion at future infinity are the
  resonances of a hyperbolic cone and were computed by the authors in
  \cite{resonances}.  The expansion treats an asymptotic regime not
  considered in the influential work of Cheeger and Taylor~\cite{CT1,
    CT2}.  The main result follows the blueprint laid out in the
  ~\cite{BVW1, BVW2} with key new elements including propagation
  estimates near the conic singularities.  The proof of the
  propagation estimates extends prior work of
  Melrose--Vasy--Wunsch~\cite{MVW} and
  Gannot--Wunsch~\cite{GW}.
\end{abstract}

\maketitle

\section{Introduction}
\label{sec:introduction}

For a given compact connected Riemannian manifold $(Z,k)$, we say that
the cone $C(Z)$ over $Z$ is the manifold
\begin{equation*}
  (0, \infty)_{r} \times Z,
\end{equation*}
equipped with the (singular) Riemannian metric
\begin{equation*}
  dr^{2} + r^{2}k.
\end{equation*}
We consider the wave equation
\begin{align}
  \label{eq:IVP}
  \left\{ \begin{array}{l}
 \Box w = (D_{t}^{2} - \lap_{C(Z)})w = 0 \in \CI_{c}(\reals \times
           C(Z)), \\
  (w, \pd[t]w)|_{t=0}  \in \CI_{c}(C(Z)) \times \CI_{c}(C(Z)), 
  \end{array}
  \right.
\end{align}
on $\reals \times C(Z)$.  Here $\lap_{C(Z)}$ represents the Friedrichs extension of the
Laplacian on $C(Z)$.

In order to simplify the statement of our main result, we introduce the
(forward) \emph{Friedlander radiation field}, which is given in terms of $s =
t-r$, $r$, and $z$ by
\begin{equation*}
  \mathcal{R}_{+}[w](s,z) = \lim_{r\to \infty} r^{(n-1)/2}w(s+r, r, z).
\end{equation*}
The function $\mathcal{R}_{+}[w]$ measures the radiation pattern seen
by a distant observer and is an explicit realization of the
Lax--Phillips translation representation as well as a generalization
of the Radon transform.  Our main theorem can then be stated in terms
of the radiation field as $s$, the ``lapse'' parameter, tends toward
infinity (a more detailed theorem is stated later as
Theorem~\ref{thm:main-thm-v2}):

\begin{theorem}
  \label{thm:main-thm}
  Suppose $w$ is a solution of the wave equation on a cone with smooth
  initial data compactly supported away from the conic singularity,
  i.e., that $w$ is a solution of equation~\eqref{eq:IVP}.
  The radiation field $\mathcal{R}_{+}[w](s,z)$ of $w$ admits an
  asymptotic expansion of the form
  \begin{equation*}
    \mathcal{R}_{+}[w](s,z) \sim \sum_{j}
    a_{j,\kappa}(z) s^{-i\sigma_{j}} 
  \end{equation*}
  as $s \to + \infty$.
\end{theorem}

In short, we find a complete asymptotic expansion for the radiation
field of a solution.  The exponents in the expansion are the
resonances of the spectral family of the Laplacian on a related
``hyperbolic cone'' and were computed in a previous
paper~\cite{resonances}.  In fact, the $\sigma_{j}$ can be computed
explicitly in terms of the eigenvalues $\mu_{j}^{2}$ of $\lap_{k}$.
Because each eigenvalue $\mu_{j}^{2}$ leads to an entire family of
resonances, it is easier to rename them $\sigma_{j,k}$ in terms of two
parameters, which we call $j$ and $k$.  Here $j$ refers to the
eigenvalue in question and $k \in \naturals = \{ 0, 1, \dots \}$.
\begin{equation}
  \label{eq:resonance-defn}
  \sigma_{j,k}  = - i \left( \frac{1}{2} + k + \sqrt{\left(
        \frac{n-2}{2}\right)^{2} + \mu_{j}^{2}}\right)
\end{equation}
provided that
\begin{equation*}
  \sqrt{\left( \frac{n-2}{2}\right)^{2} + \mu_{j}^{2}} \notin
  \frac{1}{2} + \integers.  
\end{equation*}
The resonance $\sigma_{j}$ has the same multiplicity as the eigenvalue
$\mu_{j}^{2}$ of $\lap_{k}$.  

In fact, we prove a stronger theorem showing that $w$ in fact enjoys a
joint asymptotic expansion in $r$ and $s$ of the form
  \begin{equation*}
    w \sim r^{-\frac{n-1}{2}}\sum_{j}
    \sum_{\ell = 0}^{\infty} a_{j\kappa\lambda}(z) s^{-i\sigma_{j}}
    (s/r)^{\ell}.
  \end{equation*}
A precise statement of the theorem is given at the beginning of
Section~\ref{sec:proof-theor-refth} below.  The hyperbolic cone above
is naturally realized as a boundary face of the spacetime compactification
introduced below in Section~\ref{sec:conic-geometry}; the joint
asymptotics describe the behavior near the intersection of two faces.

We note further that the hypotheses of Theorem~\ref{thm:main-thm} may
be relaxed somewhat; it is not strictly necessary that we consider the
static wave equation on a product cone; we stick to this setting for
pedagogical reasons but describe straightforward generalizations below
(see Section~\ref{sec:conic-geometry}).  Although the argument
simplifies in the product setting, the complications arising in the
general setting can be treated using more refined microlocal
techniques, though logarithmic terms might appear in the expansion.
See for instance the previous papers~\cite{BVW1, BVW2} for relaxing
the static hypothesis and Melrose--Vasy--Wunsch~\cite{MVW} to relax
the product hypothesis.  Recent work of
Yang~\cite{yang2020propagation} further connects this paper with the
work of Cheeger--Taylor by linking the scattering matrix (whose
structure can be obtained from our result) with the principal symbol
of the diffracted wave.

The results in Theorem~\ref{thm:main-thm} extend the foundational work
initiated by Cheeger--Taylor in \cite{CT1, CT2}, though our aim is
different.  Cheeger and Taylor were more interested in the propagation
of wavefront set for the wave equation on product cones; in particular
their main aim was to show the existence (and calculate the symbol) of
the diffracted wave arising from the metric singularity.  In the
process, they also found the asymptotic behavior of solutions of the
wave equation away from $\scri^{+}$; we recover their result in this
region.  Although in principle Theorem~\ref{thm:main-thm} can be
recovered using the methods of Cheeger--Taylor~\cite{CT1,CT2} provided
one could extend their asymptotic expansion uniformly to the boundary
of the light cone, we provide an alternative microlocal proof.

The novelty of this paper involves several advances on existing
technology for the study of waves in a diffractive setting.  Not only
do we essentially finish the project of Cheeger--Taylor in a fashion
that gives a complete asymptotic description, we find that cones
provide an additional class of examples where the expansion of the
radiation field can be computed explicitly using special functions
methods as in our previous work~\cite{resonances}.  In particular, the
exponents in the expansion are resonance poles of the Laplacian on the
hyperbolic cone and the coefficients in the expansion arise as
boundary values of the resonant states; the structure of the resonant
states gives insight into the structure of the scattering matrix.  A
similar observation underlies the work of Yang mentioned above~\cite{yang2020propagation}.

From a more technical perspective, we have extended a number of
microlocal tools to our setting.  In the bulk spacetime, we adapt and
extend the propagation results of Melrose--Wunsch~\cite{MW} and
Vasy~\cite{Vasy08} to our compactification.  This extension requires
putting the differential--pseudodifferential interactions at the core
of those papers on a more global footing.  

The technical heart of the paper, however, lies in our treatment of
the normal operator on the boundary in Section~\ref{sec:prop-sing}.
We extend the differential--pseudodifferential interactions to a class
of variable order Sobolev spaces on which the boundary operator is
Fredholm.  We also establish \emph{semiclassical} propagation
estimates on these spaces; to our knowledge analogous results have not
yet appeared in the literature.  The work~\cite{GW} of Gannot--Wunsch
establishes similar semiclassical estimates for conormal potentials,
which in this case can be viewed as a one-dimensional cone.

Finally, an additional technical novelty encountered is that solutions
of the wave equation are not polyhomogeneous on the final compactified
spacetime.  Indeed, they are conormal to all boundary hypersurfaces
but only polyhomogeneous at a subset of them.  To this end, we
formalize the notion of \emph{partial polyhomogeneity} in
Definition~\ref{defn:partialy-phg}.

In addition to advances in analysis, the results we obtain here have
several direct applications to important physical models.  Diffractive
systems arise naturally in physical settings where singular potentials
appear, such as in the cases of inverse square potential or the
Dirac-Coulomb system.  This framework has been adapted to study the
long time asymptotics directly for the massless Dirac-Coulomb system in 
recent work~\cite{baskin2021asymptotics}.  In addition, further
advances building upon this work have appeared in studying the
resolvent and/or scattering matrix for Laplacian on a manifold with
conic singularities in the works
\cite{hintz2020resolvents,hintz2021semiclassical,yang2020propagation},
as well as in development of propagation of singularities for conic
operators.  The partial polyhomogeneity of solutions also implies a
novel version of the so-called Price's law explored recently using
similar tools in the non-diffractive setting by
Hintz~\cite{hintz2021sharp}.  In more singular settings, the observed
decay rates change in an interesting fashion directly related to the diffractive component of the problem.  This
particular application will be explored further in a forthcoming work.

\subsection{A sketch of the proof of Theorem~\ref{thm:main-thm}}
\label{sec:sketch-proof-theorem}

To prove the main theorem, we show that solutions to
equation~\eqref{eq:IVP} are partially polyhomogeneous on a
compactification of the spacetime $\reals \times C(Z)$ and then
identify the exponents seen in the expansions.  As this proof is
somewhat involved, we provide a sketch here.

We compactify the spacetime $\reals \times C(Z)$ to a manifold with
corners we call $M$, which has two boundary hypersurfaces: one,
denoted $\mf$, corresponds to the ``boundary at infinity'', while the
other, denoted $\cf$, corresponds to the world line of the conic
singularity.  We refer the reader to Figure~\ref{fig:1-d-compact} in
the next section for a fuller picture of the geometry.

An instructive example is the case of a ``phantom cone''.  One can
view $\reals^{n}$ as a conic manifold by equipping it with polar
coordinates; in this case the link is $Z = \sphere^{n-1}$.  The
compactification to $M$ in this case can be blown down along $\cf$ to
yield the compactification of the Minkowski spacetime considered in
previous work~\cite{BVW1,BVW2}.

The proof of the main theorem roughly follows the blueprint laid out in
previous work of the first author~\cite{BVW1, BVW2}, which in turn
builds on the influential work of Vasy~\cite{vasy:microlocal}.
In particular, our aim is to reduce the problem of finding an asymptotic
expansion to the inversion of a family of Fredholm operators on $\mf$;
the residues of the poles of this family generate the terms in the
expansion.  Showing that the family is Fredholm (and that the argument
can begin) reduces to a sequence of propagation of singularities arguments.  

The forward radiation field encodes the behavior of the solution $u$
near the intersection of the future light cone with infinity (i.e.,
the face $\mf$); we denote this intersection $S_{+}$.  To find the
asymptotics of the radiation field, we therefore ultimately \emph{blow
  up} $S_{+}$ in $\mf$ to obtain a third boundary hypersurface
corresponding to ``future null infinity''.  Locally near the interior
of this new front face (denoted $\scri^{+}$), the blow-up amounts to
introducing new coordinates $\rho = (1 + t^{2} + r^{2})^{-1/2}$,
$s = t- r$, and $z$; the front face is given by $\rho = 0$.

We begin with the solution of equation~\eqref{eq:IVP}; by smoothly
cutting off the solution for $t< 0$, we consider instead the forward
solution of $\Box w = f$, where $f \in \CI_{c}(M^{\circ})$ vanishes
identically for $t< 0$.  We consider then
the function $u = \rho^{-(n-1)/2}w$ and set
\begin{equation*}
  L = \rho^{-2}\rho^{-(n-1)/2}\Box \rho^{(n-1)/2},
\end{equation*}
so that $u$ satisfies $Lu = f'$ for some other function
$f'\in \CI_{c}(M^{\circ})$ vanishing for $t<0$.  Note
that the asymptotic properties of $w$ and $u$ are linked by a simple
relationship.  A propagation of singularities argument (proved in
Section~\ref{sec:prop-sing}) shows that $u$ is conormal to $S_{\pm}$.
The conormality of the solution at the conic singularity $\cf$ is one
of the consequences of the work of
Melrose--Wunsch~\cite[Proposition~11.1]{MW}; we extend
that result to the corners $\mf \cap \cf$.

We then set $P_{\sigma} = \hat{N}(L)$ where $\hat{N}$ is the reduced
normal operator, i.e., the family of operators on $\mf$ obtained by
the Mellin transform in the normal variable $\rho$.  We set $\us$ and
$\fs$ to be the Mellin transforms of $u$ and $f'$, so that $\us$
solves
\begin{equation*}
  P_{\sigma} \us = \fs .
\end{equation*}
In general, one would expect additional correction terms, but the
dilation invariance of the model problem simplifies the argument
considerably and accounts for the absence of logarithmic terms in the
expansion in Theorem~\ref{thm:main-thm}.  We show that we can propagate regularity from the past
``radial points'' of $P_{\sigma}$ to the future ones.  Away from the
conic singularity, this argument is contained in the previous
papers~\cite{BVW1, BVW2, vasy:microlocal}; the main missing piece is
the propagation near the conic singularity
(proved in Section~\ref{sec:prop-sing}).  This argument shows that
$P_{\sigma}$ is Fredholm on variable-order Sobolev-type spaces and
$P_{\sigma}^{-1}$ has finitely many poles in any horizontal strip.  In
fact, the poles of $P_{\sigma}^{-1}$ can be identified with the
resonances of the corresponding hyperbolic cone.  

Once these pieces are in place, we can adapt the argument from the
prequel~\cite{BVW2} to prove the main theorem.  As parts of it are
somewhat more complicated in the present context, we provide a sketch
of that argument below (Section~\ref{sec:proof-theor-refth}).

Section~\ref{sec:conic-geometry} provides an introduction to the
specific geometry we consider, and
Section~\ref{sec:basics-bl-geometry} provides a brief review of the
geometry of manifolds with corners and asymptotic expansions on them.
We discuss the model operators we consider in Section
\ref{sec:operators}, then in Section~\ref{sec:pseud-oper} we present
the pseudodifferential calculi employed.
Section~\ref{sec:funct-spac-wavefr} develops the function spaces in
which the various arguments take place.
Sections~\ref{sec:prop-sing-bulk} and~\ref{sec:prop-sing} then
establish the key propagation of singularities results we need to
prove the main theorem in Section~\ref{sec:proof-theor-refth}.

\subsection{Notation}
As the proof of this paper relies on quite a few interacting differential
and pseudodifferential objects, we provide here a short list of
notational conventions employed.

We adopt the convention that $D = \frac{1}{i}\pd$.
The differential operator $\lap_{h}$ is the nonnegative Laplacian for
a Riemannian metric $h$; in a coordinate
system it is given by
\begin{equation*}
  \lap_{h} = \sum_{i,j} \frac{1}{\sqrt{h}}D_{i}\left( h^{ij}\sqrt{h}D_{j}\right).
\end{equation*}
In indexing spaces of pseudodifferential operators and Sobolev spaces,
$m$ is typically the differential order while $\ell$ represents an
order of growth or decay.

The $L^{2}$ spaces employed in this manuscript are always taken with
respect to a density induced by a metric near $x=0$; in coordinates,
these densities are given by
\begin{equation*}
  x^{n-1}\frac{d\rho}{\rho}\,dx\,\operatorname{dvol}_{\met} \text{ in
    the bulk spacetime }M, \quad \text{and} \quad
  x^{n-1}dx\,\operatorname{dvol}_{\met} \text{ on the boundary }\mf.
\end{equation*}

We aim to use the same Greek letter to denote a dual coordinate in the
cotangent bundle to a coordinate on the base; we use $\tau$ to denote
a dual to $\rho$, $\xi$ a dual to $x$, and $\zeta$ a dual to $z$.  We
use different typographical conventions to denote covectors in
different cotangent bundle constructions.  An undecorated covector
(such as $\xi$) refers to that coordinate in the standard cotangent
bundle, an underline ($\xib$) is reserved for the
$\bl$-cotangent bundle, and the subscript $\e$ is used with the
edge cotangent bundle.

\subsection*{Acknowledgments}
The authors wish to thank Semyon Dyatlov, Oran Gannot, Peter Hintz, Rafe Mazzeo, Andras Vasy and Jared Wunsch for valuable discussions. We also thank the anonymous reviewers for many helpful suggestions that led to improving the exposition of the manuscript.  DB was supported in part by
National Science Foundation (NSF) under NSF Grant DMS-1500646 and NSF CAREER Grant DMS-1654056. 
The research of JLM was supported by NSF Grant DMS-1312874 and NSF
CAREER Grant DMS-1352353. Part of this work was done while the second author was on sabbatical at Duke University and the Mittag-Leffler Institute.

\section{Conic geometry}
\label{sec:conic-geometry}

As our primary concern is the wave equation on a cone, we describe
this setting in detail.  Remark~\ref{rem:generalizations} describes
natural extensions to this setting on which versions of our main result still
hold.

Let $(Z,\met)$ be a compact, connected, $(n-1)$-dimensional
Riemannian manifold.  The metric cone $C(Z)$ over $Z$ is the manifold
\begin{equation*}
  (0,\infty)_{r}\times Z
\end{equation*}
equipped with the warped product metric
\begin{equation*}
  dr^{2} + r^{2}\met.
\end{equation*}
This metric is singular and incomplete at $r=0$; we refer to the
natural boundary $\{ 0\} \times Z$ as the \emph{cone
  point}.\footnote{We regard the conic singularity as being purely
  metric; one can think of the underlying manifold as having been
  previously resolved by blowing up a conic singularity.}

Our main result concerns solutions of the wave equation on the
spacetime $M^{\circ} = \reals_{t}\times C(Z)$, which is equipped with
the Lorentzian metric
\begin{equation*}
  g = -dt^{2} + dr^{2} + r^{2}\met.
\end{equation*}

We may regard $M^{\circ}$ as the interior of a compact manifold with
corners.  For clarity, we first describe this compactification in the
$(1+1)$-dimensional setting (i.e., when $Z$ is a single point) even
though Theorem~\ref{thm:main-thm} is trivial in this case.

\begin{figure}
  \centering
  \includegraphics{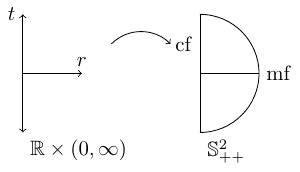}
  \caption{The compactification of $\reals \times (0,\infty)$ to
    $\sphere^2_{++}$}
  \label{fig:1-d-compact}
\end{figure}

We compactify $\reals_{t}\times (0,\infty)_{r}$ by stereographic
projection to a quarter-sphere $\sphere^{2}_{++}$ as depicted in
Figure~\ref{fig:1-d-compact}.  In other words, the map $\reals_{t}\times
(0,\infty)_{r}\to \sphere^{2}\subset \reals^{3}$ given by
\begin{equation*}
  (t,r) \mapsto \frac{(t,r,1)}{\sqrt{1 + t^{2}+r^{2}}}
\end{equation*}
sends $M^{\circ}$ to the interior of the quarter-sphere given by
\begin{equation*}
  \sphere^{2}_{++} = \left\{ (z_{1},z_{2},z_{3}) \in
    \sphere^{2}\subset \reals^{3} \mid z_{2}\geq 0, z_{3}\geq 0\right\}.
\end{equation*}
The quarter-sphere $\sphere^{2}_{++}$ is a manifold with corners and
has two boundary hypersurfaces defined by the boundary defining
functions $z_{2}$ and $z_{3}$.  We let $\cf$ (or the \emph{conic
  face}) be the hypersurface defined by the function
\begin{equation*}
  z_{2} = \frac{r}{\sqrt{1+t^{2}+r^{2}}}
\end{equation*}
and we let $\mf$ (or the \emph{main face}) be the face defined by
\begin{equation*}
  z_{3} = \frac{1}{\sqrt{1+t^{2}+r^{2}}}.
\end{equation*}
The boundary hypersurface $\mf$ plays an outsized role in the
manuscript and is often referred to as $X$ when considered on its own.

\begin{figure}
  \centering
  \includegraphics{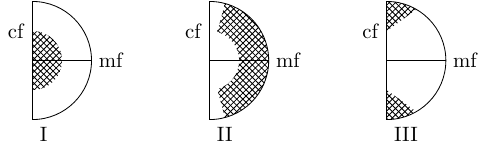}
  \caption{Regions I, II, and III in $\sphere^2_{++}$}
  \label{fig:regions}
\end{figure}

Having defined the smooth structure of this compactification, it is
often convenient to work with other equivalent boundary defining
functions in different regions.  We define regions I, II, and III (the
shaded regions in Figure~\ref{fig:regions}) as follows: We let region
I denote a fixed neighborhood in $\sphere^{2}_{++}$ bounded away from
$\mf$; region II is a neighborhood of $\mf$ bounded away from $\cf$;
finally, region III is a neighborhood of the corners $\mf\cap \cf$.
For concreteness, we can take region I to be given by $\{\abs{t}, r \leq
10\}$, region II to be $\{ r \geq 2, r \geq \abs{t}/2\}$, and region
III to be $\{ \abs{t} \geq 2, \abs{t}\geq r/2\}$.  Note that region
III has two connected components; we typically work with only one
component at a time.

In defining the Mellin transform below, it is useful to have a fixed
boundary defining function for $\mf$.  For this purpose, we let $\rho$
denote a defining function for $\mf$ that is equal to $\frac{1}{t}$
for $t/r > 1/2$ and equal to $\frac{1}{-t}$ for $t/r < -1/2$.

We now describe several convenient boundary defining functions valid
in the different regions.  For notational convenience, we always use
$\rho$ (or $\rhot$) to denote a defining function for $\mf$ and $x$ to denote a
defining function for $\cf$.  In region I, (where we are bounded away
from $\mf$), it is convenient to take $x=r$, while in region II (where
we are bounded away from $\cf$), we can take $\rhot = 1/r$.  Finally,
in region III (the source of most of the new technical work in this
manuscript), it is typically convenient to take $\rho = \pm 1/t$ and
$x = r / \abs{t}$.  \emph{Because polyhomogeneity is independent of
  the choice of equivalent boundary defining functions, one can typically
  use whichever boundary defining functions are most convenient at the
  time.}

On the $(1+1)$-dimensional Lorentzian manifold $\sphere^{2}_{++}$, we
employ coordinate systems specialized to the different regions.  In
region I, we employ $x=r$ and use coordinates $(t, x)$; the Lorentzian
metric here has the familiar form
\begin{equation*}
  -dt^{2} + dx^{2}.
\end{equation*}
In region II, the metric has the form of a short-range asymptotically
Minkowski metric as employed by the first author and
collaborators~\cite{BVW1}; we use $(\rhot, y)$ as coordinates, where
$\rhot = 1/r$ and $y = t/r$.  The metric in this coordinate system has
the form
\begin{equation*}
  - \frac{dy^{2}}{\rhot^{2}} + 2y \frac{dy}{\rhot}\frac{d\rhot}{\rhot^{2}}
  + (1-y^{2})\frac{d\rhot^{2}}{\rhot^{4}}.
\end{equation*}
Near the corner (region III), in terms of $(\rho, x)$ the metric has
the form
\begin{equation*}
  - (1-x^{2}) \frac{d\rho^{2}}{\rho^{4}} - 2x
  \frac{dx}{\rho}\frac{d\rho}{\rho^{2}} + \frac{dx^{2}}{\rho^{2}}.
\end{equation*}

For the more general case of $M^{\circ}= \reals \times C(Z)$, we take
$M$ to be the closure of the image of $M^{\circ}$ under the map
$\reals \times (0,\infty) \times Z \to \sphere^{2}\times Z$ given
by
\begin{equation*}
  (t,r,z) \mapsto \left( \frac{(t,r,1)}{\sqrt{1+t^{2}+r^{2}}}, z\right).
\end{equation*}
In other words, we take $M = \sphere^{2}_{++}\times Z$ to be the
compactification of $M^{\circ}$ to a manifold with corners.

In region I, the metric is the spacetime metric on a conic manifold
studied by Melrose--Wunsch~\cite{MW} (and later by
Melrose--Vasy--Wunsch~\cite{MVW}).  In region II, $g$ has the form
\begin{equation}
  \label{eq:metric-in-r2}
  g = -\frac{dy^{2}}{\rhot^{2}} + 2y \frac{dy}{\rhot}
  \frac{d\rhot}{\rhot^{2}} + (1-y^{2})\frac{d\rhot^{2}}{\rhot^{4}} + \frac{\met}{\rhot^{2}},
\end{equation}
which is again a short-range asymptotically Minkowski metric (though
written in somewhat different coordinates than those in~\cite{BVW1}).

Near the corner (region III), the metric has the form
\begin{equation}
  \label{eq:metric-in-r3}
  g = - (1-x^{2})\frac{d\rho^{2}}{\rho^{4}} - 2x \frac{dx}{\rho}
  \frac{d\rho}{\rho^{2}} + \frac{dx^{2}}{\rho^{2}} + \frac{x^{2}\met}{\rho^{2}}.
\end{equation}
This metric is a hybrid of a Lorentzian scattering metric (in that it
is built from $1$-forms of the type $d\rho / \rho^{2}$ and
$\alpha/\rho$) and a conic type metric (in that it degenerates as
$x\to 0$).

\begin{remark}
  \label{rem:generalizations}
  There are a number of natural extensions to the product cone setting
  that require little additional work.  All of the results and proofs
  in this manuscript (other than the explicit characterization of
  exponents) apply to the setting where $g$ is a Lorentzian metric on
  $M = \sphere^{2}_{++}\times Z$ that is
  \begin{enumerate}
  \item a spacetime conic metric (so that the results of
    Melrose--Wunsch~\cite{MW} apply) in region I,
  \item a (long-range or short-range) asymptotically Minkowski metric
    in region II, and 
  \item a hybrid in region III.  In other words, in region III, we
    demand that $g$ is built out of $d\rho^{2}/\rho^{2}$, $dx/\rho$,
    and $dz/\rho$ and that its leading order behavior as $x\to 0$ (in
    terms of these objects) is
    \begin{equation*}
      -\frac{d\rho^{2}}{\rho^{4}} + \frac{dx^{2}}{\rho^{2}} + \frac{x^{2}\met}{\rho^{2}}.
    \end{equation*}
  \end{enumerate}
\end{remark}

\subsection{The radiation field blow-up}
\label{sec:radiation-field-blow}

In this section we recall from previous work~\cite{BVW1,BVW2} the
construction of the manifold with corners on which the radiation field
naturally lives.

Consider the submanifold $S = \{ \rho = 0, y = \pm 1\}$ in region II ($S$ is
given by $\{ \rho = 0, x = 1\}$ in region III).  This submanifold
naturally splits into two components according to whether $\pm t > 0$
near the component.  We use $S_{\pm}$ to denote these two pieces; they
split the complement of $S$ in $\mf$ into three connected components.
We use $C_{0}$ to denote those points in $\mf$ where $y \in (-1,1)$,
while the subset of $\mf$ in region III where $x < 1$ has two
components, denoted $C_{\pm}$ according to whether $\pm t > 0$
nearby.

\begin{figure}
  \centering
  \includegraphics{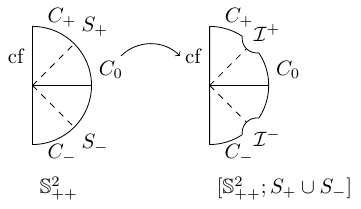}
  \caption{A schematic view of the radiation field blow-up.  The lapse
    function $s$ increases along $\scri^+$ towards $C_+$.}
  \label{fig:rf-blowup}
\end{figure}

We now \emph{blow up} $S$ in $M$ by replacing it with its inward
pointing spherical normal bundle.\footnote{The reader may wish to
  consult Melrose's book~\cite{Melrose:APS} for more details of the
  blow-up construction.}  In the product cone setting, this is
equivalent to blowing up a pair of points in $\sphere^{2}_{++}$ and
then taking the product with $Z$.  This process replaces $M$ with a
new manifold $\overline{M} = [ M ; S]$ on which polar coordinates
around the submanifold are smooth; the smooth structure of this
manifold with corners depends only on the submanifold $S$ (and not on
the particular choice of defining functions).  The blow-up is equipped
with a natural blow-down map $\overline{M}\to M$; this map is a
diffeomorphism on the interior. Figure~\ref{fig:rf-blowup} depicts
this blow-up construction.

The new space $\overline{M}$ is again a manifold with corners and has
six boundary hypersurfaces: the closure of the lifts of the interiors
of $C_{0}$ and $C_{\pm}$ to $\overline{M}$, which are again denoted by
$C_{0}$ and $C_{\pm}$; the lift of $\cf$, again denoted $\cf$, and two
new boundary hypersurfaces consisting of the pre-images of $S_{\pm}$
under the blow-down map.  These two new hypersurfaces are called
future/past null infinity and denoted by $\scri^{\pm}$.  Moreover,
$\scri^{\pm}$ is naturally a (trivial) fiber bundle over $S_{\pm}$
with fibers diffeomorphic to intervals.  Indeed, the interior of each
fiber is naturally an affine space (i.e., $\reals$ acts by
translations, but there is no natural origin).  In terms of $y$ and
$\rho$, the fibers of the interior of $\scri^{\pm}$ in $\overline{M}$
can be identified with $\reals \times Z$ via the coordinate
$s = \pm (y \mp 1)/\rho$.  In other words, $s=t-r$ provides a
coordinate along $\scri^{+}$ and $s = r+t$ is a coordinate along
$\scri^{-}$.

In our setting, Friedlander's argument~\cite{Friedlander:1980,
  Friedlander:2001} shows that for solutions $w$ of the wave equation
$\Box_{g}w = 0$ with smooth, compactly supported initial data, the
restriction
\begin{equation*}
  \mathcal{R}_{\pm}[w](s,z) = \rho^{-\frac{n-1}{2}}w \rvert_{\scri^{\pm}}
\end{equation*}
is well-defined and smooth.  This is \emph{Friedlander's radiation
  field}.\footnote{Note that our definition differs from Friedlander's
  by the absence of a derivative.}

\section{Basics of $\bl$-geometry}
\label{sec:basics-bl-geometry}

The main results in this paper require an understanding of the
interaction between Melrose's $\bl$-calculus and differential
operators on cones.  In light of the compactification described above,
we begin by recalling results about analysis on manifolds with
corners.  Some of the discussion in the next few sections is adapted
from prior work of the first author~\cite{BVW1, BVW2}, while a more
thorough discussion of $\bl$-geometry can be found in Melrose's
book~\cite[Chapter 4]{Melrose:APS}.  In the context of manifolds with
corners, we refer the reader to Melrose's unpublished
book~\cite{Melrose:MWC} and to Vasy's work~\cite{Vasy08}.

Throughout the paper we let $M$ denote a compact $(n+1)$-dimensional
manifold with corners and $X$ a compact $n$-dimensional manifold with
boundary.  A function $\rho \in C^{\infty}(M)$ is a boundary defining
function for a boundary hypersurface $H$ of $M$ if $\rho$ vanishes
simply at $H$ and is non-vanishing elsewhere.  A codimension $k$
corner is the intersection of $k$ boundary hypersurfaces of $M$.  Near
a codimension $k$ corner $H_{1}\cap \dots \cap H_{k}$, we may use
\begin{equation*}
  \left(\rho_{1}, \dots, \rho_{k}, y\right) \in [0,1)^{k} \times
  \reals^{n+1-k}
\end{equation*}
as coordinates on $M$, where $\rho_{j}$ is a boundary defining
function for $H_{j}$ and $y$ are coordinates along the corner
$H_{1}\cap \dots \cap H_{k}$.

As our main applications involve corners of codimension no greater
than two, we now specialize to that case.  We assume now that $M$ has
corners of codimension two and that $\rho$ and $x$ are boundary
defining functions (to keep consistent with notation above) in a
neighborhood of a codimension two corner.  Further keeping consistent
with our notation, we use $z$ to denote the remaining coordinates.

The space of $\bl$-vector fields on $M$, denoted $\mathcal{V}_{\bl}(M)$, is the
space of smooth vector fields on $M$ tangent to $\pd M$.  Near a
codimension $2$ corner $\{\rho = x = 0\}$,
$\mathcal{V}_{\bl}(M)$ is spanned over $C^{\infty}(M)$ by the vector
fields $\rho\pd[\rho], x\pd[x], \pd[z]$.
The vector field $\rho\pd[\rho]$ is called the $\bl$-normal
vector field to the boundary hypersurface $\{\rho=0\}$ and is independent
of choice of coordinate system as an element of
$\mathcal{V}_{\bl}(M)/\rho\mathcal{V}_{\bl}(M)$.

In fact, $\mathcal{V}_{\bl}(M)$ is a Lie algebra and is the space of
smooth sections of a vector bundle (called the $\bl$-tangent bundle)
$\Tb M$ over $M$.  The sections of its dual bundle $\Tbstar M$ are
locally spanned near a codimension $2$ corner over $C^{\infty}(M)$ by
the $1$-forms $d\rho / \rho, dx/x$, and
$dz$.  

The $\bl$-cotangent bundle $\Tbstar M$ is equipped with a canonical
$1$-form, which can be written
\begin{equation}
  \label{eq:canon-1-form}
  \taub \frac{d\rho}{\rho} + \xib \frac{dx}{x} + \zetab \cdot dz
\end{equation}
in local coordinates near a codimension $2$ corner.  The fiber
compactification $\Tbstarbar M$ of $\Tbstar M$ is given by radially
compactifying each fiber.  A defining function for the ``boundary at
infinity'' of a fiber is given by
\begin{equation*}
  \nu = \left( \taub^{2} + \xib^{2} + \abs{\zetab}^{2}\right)^{1/2},
\end{equation*}
and near infinity we may use
\begin{equation*}
  \nu, \widehat{\taub} = \nu \taub, \widehat{\xib} = \nu \xib ,
  \widehat{\zetab} = \nu \zetab
\end{equation*}
as a redundant set of coordinates on each fiber near $\{ \nu = 0, \rho
= 0, x = 0\}$.\footnote{Strictly speaking, we should regard
  $(\widehat{\taub}, \widehat{\xib} ,\widehat{\zetab})\in \sphere^{n}$
  and then regard $(\nu, \widehat{\taub}, \widehat{\xib}
  ,\widehat{\zetab})$ as ``polar coordinates'' near infinity.}  We let
$\Sbstar M$ denote the boundary at infinity of $\Tbstarbar M$, i.e.,
$\{ \nu = 0\}$.

The $\bl$-cotangent bundle further inherits a canonical symplectic
structure where the symplectic form is given by the exterior
derivative of the canonical $1$-form.  (In other words, the natural
symplectic structure on $T^{*}M$ extends to $\Tbstar M$.)  If we write
covectors in $\Tbstar M$ in local coordinates as
\begin{equation*}
  \taub \frac{d\rho}{\rho} + \xib \frac{dx}{x} + \zetab\cdot dz,
\end{equation*}
then the symplectic form is given by
\begin{equation}
  \label{eq:sympl-form}
  d\taub \wedge \frac{d\rho}{\rho} + d\xib \wedge \frac{dx}{x} +
  d\zetab\wedge dz.
\end{equation}

As $\mathcal{V}_{\bl}(M)$ is a Lie algebra, we also consider its
universal enveloping algebra, denoted $\Diffb^{*}(M)$.  Near the
codimension $2$ corner defined by $\{ \rho = x = 0\}$, an operator $A
\in \Diffb^{m}(M)$ has the form
\begin{equation}
  \label{eq:form-of-diffb}
  A = \sum_{j + k + \abs{\alpha}\leq m}a_{jk\alpha} (\rho, x, z)
  \left( \rho D_{\rho}\right)^{j}\left( x D_{x}\right)^{k}D_{z}^{\alpha},
\end{equation}
where $a_{jk\alpha}\in C^{\infty}(M)$.  The principal symbol of such
an operator is given by
\begin{equation*}
  \sigmab (A) = \sum_{j+k+\abs{\alpha}=m}a_{jk\alpha}(\rho,x,z) \taub^{j}\xib^{k}\zetab^{\alpha}.
\end{equation*}

The semiclassical version of $\Diffb^{m}(M)$, denoted
$\Diffbh^{m}(M)$, is similarly defined with a parametric dependence on
a small parameter $h> 0$.  In local coordinates, an operator $A \in
\Diffbh^{m}(M)$ has the form
\begin{equation}
  \label{eq:form-of-diffbh}
  A = \sum_{j + k + \abs{\alpha}\leq m}a_{jk\alpha} (\rho, x, z;h)
  \left( h\rho D_{\rho}\right)^{j}\left( hx D_{x}\right)^{k}(hD_{z})^{\alpha}
\end{equation}
where $a_{jk\alpha}\in C^{\infty}(M)$ are bounded in $h$.  In fact we
require $\Diffbh^{*}$ only in the context of the manifold with
boundary $X$.  The semiclassical principal symbol of such an operator
captures the leading order behavior, i.e., up to $h\Diffbh^{m-1}(M)$.

While the principal symbol of a differential operator captures its
high-frequency behavior, it fails to characterize the asymptotic
behavior at the boundary.  At each boundary face, there is a
dilation-invariant model operator, called the \emph{normal operator}
that captures this behavior.  We require this operator only at the
face given by $\{\rho = 0\}$, where it is obtained by freezing the
coefficients of $\rho D_{\rho}$, $x D_{x}$, and $D_{z}$ at $\rho =
0$.  In other words, $N(A) \in \Diffb^{m}([0,\infty) \times \{ \rho =
0\})$ and is given by
\begin{equation}
  \label{eq:normal-op-defn}
  N(A) = \sum_{j + k + \abs{\alpha} \leq m} a_{jk\alpha}(0,x,z) (\rho
  D_{\rho})^{j} (x D_{x})^{k} D_{z}^{\alpha}.
\end{equation}

Just as the Fourier transform is central to the study of approximately
translation-invariant operators, the \emph{Mellin transform} is useful
in the study of approximately dilation-invariant operators.  For the
main application of this paper, we need only the Mellin transform
associated to a single boundary hypersurface $H= \{ \rho = 0\}$.
Suppose $u$ is a distribution on $M$ suitably localized near the
boundary hypersurface $H$ defined by $\rho$.  The Mellin transform of
$u$ associated to $H$ is defined by
\begin{equation*}
  \us = \mathcal{M} u (\sigma, x,z) = \int_{0}^{\infty}\chi(\rho) u(\rho, x,
  z) \rho^{-i\sigma-1}\,d\rho,
\end{equation*}
where $\chi$ is a smooth compactly supported function that is equal to
$1$ near $\rho = 0$.

The Mellin conjugate of the operator $N(A)$ is known as the
\emph{reduced normal operator}.\footnote{We require this construction
  only for differential operators, though it extends to
  $\bl$-pseudodifferential operators as well.}  For $N(A)$ given by
the formula~\eqref{eq:normal-op-defn} above, the reduced normal
operator is the family of operators on the boundary hypersurface $H$
given by
\begin{equation}
  \label{eq:reduced-normal-op-defn}
  \widehat{N}(A) = \sum_{j+k+\abs{\alpha} \leq
    m}a_{jk\alpha}(0,x,z)\sigma^{j}(x D_{x})^{k}D_{z}^{\alpha}.
\end{equation}

The Mellin transform is particularly useful in the study of asymptotic
expansions in powers of $\rho$ and $\log \rho$.  We first discuss the
case where $M$ has only a single boundary hypersurface, i.e., when $M$
is a manifold with boundary.  In particular, we recall from
Melrose~\cite[Section 5.10]{Melrose:APS} that if $u$ is a distribution
on a manifold with boundary, we say that \emph{$u$ is polyhomogeneous
  with index set $E$} if and only if $u$ is conormal to $\pd M$ (in
particular, $u$ is smooth away from the boundary), and
\begin{equation*}
  u \sim \sum_{(z,k)\in E} \rho^{iz} (\log \rho)^{k}a_{zk},
\end{equation*}
where $a_{zk}$ are smooth functions on $\pd M$.  Here the expansion
should be interpreted as an asymptotic series as $\rho \to 0$ and $E$
is an \emph{index set} and therefore must satisfy\footnote{We have
  adopted the index set convention of Melrose's unpublished
  book~\cite{Melrose:MWC} rather than the other
  reference~\cite{Melrose:APS} to remain consistent with the first
  author's prior work~\cite{BVW1, BVW2}.}
\begin{itemize}
\item $\ds E \subset \complexes \times \{ 0, 1, 2, \dots \}$,
\item $E$ is discrete,
\item if $(z_{j}, k_{j}) \in E$ with $\abs{(z_{j},k_{j})}\to \infty$,
  then $\Im z_{j} \to -\infty$, 
\item if $(z,k) \in E$, then $(z,l) \in E$ for all $l = 0, 1, \dots,
  k-1$, and 
\item if $(z,k) \in E$, then $(z-ij, k)\in E$ for all $j=1,2,\dots$.
\end{itemize}
We refer the reader to Melrose's book~\cite[Section 5.10]{Melrose:APS}
for a discussion of the naturality of these conditions.  As an
example, the functions that are smooth up to $\pd M$ are
polyhomogeneous with index set $E = \{ (-ij, 0) : j = 0, 1, 2,
\dots\}$.

Polyhomogeneous distributions are characterized in two different ways:
by the Mellin transform and by the application of scaling (or
\emph{radial}) vector fields.  To see the former, we recall a
characterization of this space given by Melrose~\cite[Proposition
5.27]{Melrose:APS}.  For a given index set $E$, a distribution $u$ is
polyhomogeneous with index set $E$ if and only if its Mellin transform
is meromorphic with poles of order $k$ only at points $z$ for which
$(z,k-1)\in E$ (together with appropriate decay estimates in
$\sigma$).

Alternatively, we may test for polyhomogeneity by using radial vector
fields.  Let $R$ denote the radial vector field $\rho D_{\rho}$.  We
characterize a polyhomogeneous distribution $u$ with index set $E$ by
the requirement that for all $A$, there is a $\gamma_{A}$ with
$\gamma_{A}\to +\infty$ as $A\to +\infty$ so that
\begin{equation}
  \label{eq:phg-charac-b1}
  \left( \prod_{(z,k)\in E, \Im z > -A} (R-z)\right) u \in
  \rho^{\gamma_{A}}\Hb^{\infty}(M),
\end{equation}
where $\Hb^{\infty}(M)$ denotes the space of distributions conormal to
the boundary.

Our main theorem concerns polyhomogeneity at two boundary
hypersurfaces on a manifold with codimension $2$ corners.  We apply
this characterization below to the manifold $\overline{M}$, which has six
boundary hypersurfaces $\cf, C_{\pm}, C_{0}$, and $\scri^{\pm}$.  The
distributions we consider vanish identically near $\scri^{-},
C_{-}$, and $C_{0}$, so there are three hypersurfaces of interest.

In the characterization that follows, we let $H_{1} = \scri^{+}$,
$H_{2} = C_{+}$, and $H_{3} = \cf$ denote the relevant hypersurfaces;
for now we let $\rho_{1}$ define $H_{1}$, $\rho_{2}$ define $H_{2}$,
and $x$ define $H_{3}$.  We now define the space of \emph{partially
  polyhomogeneous} distributions with index sets $\cE =
(E_{1},E_{2})$.  
\begin{definition}
  \label{defn:partialy-phg}
  A distribution $u$ lies in $\pphg^{\cE}(\overline{M})$, the space of partially
  polyhomogeneous distributions with index sets $\cE$, if $u$ is
  conormal to all boundary hypersurfaces, and, for each $j=1,2$, we
  have
  \begin{equation*}
    u \sim \sum_{(z,k)\in E_{j}}a_{jzk}\rho^{iz}(\log \rho)^{k} \quad
    \text{mod } \rho_{j}^{\infty}\rho_{3-j}^{-A}x^{-B}\Hb^{\infty}(\overline{M}),
  \end{equation*}
  where $A$ is some fixed number greater than $\sup \{ \Im z \mid
  (z,k) \in E_{j}, j=1,2\}$, $B$ is some fixed number, and $a_{jzk}$
  are smooth at the hypersurface defined by $\rho_{j}$, conormal at
  $H_{3}$, and polyhomogeneous (with index set $E_{3-j}$) at the other one.
\end{definition}

When testing for (partial) polyhomogeneity at multiple boundary
hypersurfaces, it suffices to test individually at each one with
uniform estimates at the others.  This result is due independently to
Mazzeo~\cite[Appendix]{Economakis} and Melrose~\cite[Chapter
4]{Melrose:MWC} and is a consequence of a characterization by multiple
Mellin transforms.  In particular, we appeal to the following
proposition.
\begin{proposition}[cf. Mazzeo, Melrose]
  \label{prop:2-step-phg}
  Let $R_{j}$ denote $\rho_{j}D_{\rho_{j}}$, the radial vector field
  at the boundary hypersurfaces defined by $\rho_{j}$.  For $\cE =
  (E_{1}, E_{2})$, a distribution $u$ lies in $\pphg^{\cE}(\overline{M})$ if and only if
  it is conormal to all boundary hypersurfaces and for each $j=1,2$
  there are fixed weights $\alpha_{j}, \beta_{j}$ and for all $A$, there is a
  $\gamma_{j,A}$ with $\gamma_{j,A}\to +\infty$ as $A\to +\infty$, so
  that
  \begin{equation*}
    \left( \prod_{(z,k)\in E_{j}, \Im z > -A}(R_{j}-z)\right)u \in
    \rho_{j}^{\gamma_{j,A}}\rho_{3-j}^{\alpha_{j}}x^{\beta_{j}} \Hb^{\infty}(\overline{M}).
  \end{equation*}
\end{proposition}
In other words, applying the test~\eqref{eq:phg-charac-b1} above at
the boundary hypersurface $H_{j}$ defined by $\rho_{j}$ improves the
decay at $H_{j}$ at no cost to the growth/decay at the other
hypersurfaces.  Note that there is no requirement that the
coefficients be polyhomogeneous; their joint polyhomogeneity at
$H_{1}\cap H_{2}$ follows automatically when the condition is imposed
individually at $H_{1}$ and $H_{2}$.

\section{The operators $L$ and $\Ps$}
\label{sec:operators}

Friedlander's argument for the existence of the radiation field
motivates the definition below of the operator
\begin{equation*}
  L = \rho^{-2- (n-1)/2}\Box_{g} \rho^{(n-1)/2},
\end{equation*}
and its reduced normal operator $\Ps = \widehat{N}(L)$.  Because
changing the boundary defining functions by a smooth non-vanishing
multiple changes $L$ and $\Ps$ by a lower order term, we freely work
with whichever forms of the boundary defining functions are most
convenient.

For later reference, we record the forms of the operators $L$ and
$\Ps$ in region III, where the metric has the form as in
equation~\eqref{eq:metric-in-r3}.\footnote{We do not use the explicit
  form of the operators in region II and instead appeal standard
  hyperbolic propagation estimates as in previous work~\cite{BVW1}
  there.}  Indeed, we write (using $D = \frac{1}{i}\pd$):
\begin{align*}
  L &= \left( \rho D_{\rho} + x D_{x} \right)^{2} - ni\left( \rho
      D_{\rho} + x D_{x}\right) - D_{x}^{2} + \frac{(n-1)i}{x}D_{x} -
      \frac{1}{x^{2}}\lap_{\met} - \frac{n^{2}-1}{4}, \\
  \Ps &= \left( x D_{x} + \sigma\right) - ni \left( x D_{x} +
        \sigma\right) - D_{x}^{2} + \frac{(n-1)i}{x}D_{x} -
        \frac{1}{x^{2}}\lap_{\met} - \frac{n^{2}-1}{4}.
\end{align*}

In the main propagation results of Sections~\ref{sec:prop-sing-bulk}
and~\ref{sec:prop-sing}, we require an understanding of the Hamilton
flow of the principal symbols of the operators $L$ and $\Ps$.  We
start by describing this flow near $\cf$ (in $M$) and near $\pd X =
\mf \cap \cf$ (in $X = \mf$).

\subsection{Broken bicharacteristics for the operator $L$}
\label{sec:operator-l}

We now aim to describe set of broken bicharacteristics
along which singularities may propagate.  Perhaps the shortest path to
their characterization involves the \emph{edge cotangent bundle},
which we describe shortly.  Moreover, the propagation arguments in the
bulk spacetime $M$ in Section~\ref{sec:prop-sing-bulk} require
commuting $\bl$-pseudodifferential operators through the differential
operators naturally associated to the conic metric.  It is therefore
convenient to introduce a small amount of the edge calculus machinery
(namely, the bundles and the differential operators) introduced by
Mazzeo~\cite{rafe1991elliptic}.  We specialize our description to the specific
setting in which we work, though the calculus applies in much more
general settings.  In an abuse of notation, we use the term ``edge''
to refer to objects that behave as edge objects at $\cf$ and as
$\bl$-objects at $\mf$.  The reader wishing to skip this section need
only note that the space of edge differential operators $\Diffe^{*}$
and the compressed characteristic set $\Sigmadot$ are referred to
later.

Our use of the edge machinery is limited to a neighborhood of the
boundary hypersurface $\cf$ corresponding to the conic singularity.
This boundary hypersurface is the total space of a trivial fiber
bundle:
\begin{equation*}
  \begin{tikzcd}
    Z \arrow[r, dash] & \cf \arrow[d] \\
    & I
  \end{tikzcd}
\end{equation*}
Here $I$ is a compactification of $\reals$ to an interval; $t$ is
locally a coordinate on the interior of $I$ while $\rho$ provides a
coordinate near each endpoint of $I$.

The set of edge-vector fields, typically denoted $\mathcal{V}_{\e}$,
consists of those $\bl$-vector fields that are tangent to the leaves
of the fibration.  In local coordinates $(\rho, x, z)$ where $x$ is
the boundary defining function for $\cf$ and $z$ is a coordinate along
$Z$, $\mathcal{V}_{\e}$ is spanned over $C^{\infty}$ by
\begin{equation*}
  x\pd[x], \quad x\rho \pd[\rho], \quad \text{and} \quad \pd[z_{j}].
\end{equation*}
The Lie algebra $\mathcal{V}_{\e}$ is the space of smooth sections of
a vector bundle (called the $\e$-tangent bundle) $\Te M$ over
$M$.\footnote{Strictly speaking, as a global object, we are
  considering a mixed edge-$\bl$-tangent bundle, but our arguments are
  essentially local so we do not stress this point.}  Its dual is the
$\e$-cotangent bundle $\Testar M$.

We let $\Diffe^{*}(M)$ denote the universal enveloping algebra of
$\mathcal{V}_{e}(M)$.  An element $A\in \Diffe^{m}(M)$ near $\mf \cap
\cf$ has the form
\begin{equation*}
  A = \sum_{j+k+\abs{\alpha}\leq m}a_{jk\alpha}(\rho,x,z) (x\rho
  D_{\rho})^{j}(x D_{x})^{k}D_{z}^{\alpha},
\end{equation*}
where the $a_{jk\alpha}$ are smooth on $M$.  The operator $L$ is an
element of $x^{-2}\Diffe^{2}(M)$; this relationship is exploited below
in Section~\ref{sec:prop-sing-bulk}.

Canonical coordinates on $\Testar M$ induced by the coordinates
$(\rho, x, z)$ are $(\rho, x, z, \taue, \xie, \zetae)$, which
corresponds to writing covectors as
\begin{equation*}
  \taue \frac{d\rho}{x\rho} + \xie \frac{dx}{x} + \zetae \cdot dz.
\end{equation*}
One then obtains a bundle map $\pi : \Testar M \to \Tbstar M$ given in
these coordinates by
\begin{equation*}
  \pi \left( \rho, x, z, \taue, \xie, \zetae\right) = \left( \rho , x,
    z, \taub = \taue, \xib = x\xie, \zetab = x\zetae\right).
\end{equation*}
In other words, the map $\pi$ is given by $\omega \mapsto x\omega$,
which is an isomorphism $\Testar M \to \Tbstar M$ away from $x=0$.

Away from $x=0$, the bicharacteristics (in this case lifts of
geodesics to the $\bl$-cotangent bundle) of $L$ are the integral
curves of the $\bl$-Hamilton vector field of the $\bl$-principal
symbol of $L$.  As $(M,g)$ is incomplete owing to the conic
singularity of $C(Z)$, we must clarify what we mean by
bicharacteristics that hit the cone point.  As we are interested in
wave equations, we restrict our attention to \emph{null
  bicharacteristics}, i.e., those lying in the characteristic set of
$L$.

We define now the \emph{compressed cotangent bundle} by
\begin{equation*}
  \Tbdotstar M = \pi (\Testar M) / Z, \quad \dot{\pi} : \Testar M \to
  \Tbdotstar M,
\end{equation*}
where the quotient by $Z$ acts only over the boundary; the topology is
given by the quotient topology.  Observe that $\Tbdotstar_{\cf} M$ can
be identified with $\Tbstar I$; in terms of coordinates $(\rho, x, z,
\taue, \xie, \zetae)$ on $\Testar M$, $\pi (\Testar_{\cf}M)$ is given
by points of the form $(\rho, 0, z, \taub, 0, 0)$.  After the
quotient, $\rho$ and $\taub$ provide coordinates on $\Tbstar I$.

In an abuse of notation (but following
Melrose--Vasy--Wunsch~\cite[Section 7]{MVW}), we introduce
\begin{align*}
  \pi (\Sestar M) &= \left( \pi (\Testar M) \setminus 0 \right) /
                    \reals^{+} \subset \Sbstar M, \\
  \dot{\pi} \left( \Sestar M\right) &= \left( \dot{\pi}(\Testar M) \setminus
  0\right) / \reals^{+} \subset \Sbdotstar M,
\end{align*}
where $\Sbstar M$ and $\Sestar M$ are quotients of their respective
cotangent bundles by the natural scaling action and $\Sbdotstar M =
\Sbstar M / Z$ with the quotient acting over $\cf$.

We now observe that $x^{2}L\in \Diffe^{2}(M)$; near $\mf \cap \cf$,
its edge-principal symbol is
\begin{equation*}
  \sigma_{\e}(x^{2}L) = (\taue + x\xie)^{2} - \xie^{2} -
  \abs{\zetae}^{2}
\end{equation*}

As $\cf$ is noncharacteristic for $L$, nonzero covectors in the
edge-characteristic set of $x^{2}L$ (i.e., the vanishing set for
$\sigma_{e}(x^{2}L)$) are mapped to nonzero covectors by $\pi$ and
$\dot{\pi}$.  We can thus define the \emph{compressed characteristic set}
\begin{equation*}
  \Sigmadot = \dot\pi (\Sigma), 
\end{equation*}
where $\Sigma\subset \Sestar M$ is the edge-characteristic set of
$x^{2}L$.  Over $x=0$, $\Sigmadot = \Sbdotstar_{\cf}M$, i.e.,
\begin{equation*}
  \Sigmadot \rvert_{\cf} = \left\{ (\rho, x=0, z, \taub, 0, 0) \mid
    \tau \neq 0, z\in Z\right\} / Z.
\end{equation*}
In the parlance of Melrose--Vasy--Wunsch, all of the points of $\Sigmadot$ lying
over $\cf$ are hyperbolic.

There are many equivalent and nearly-equivalent definitions of
generalized broken bicharacteristics (see, e.g.,
Melrose--Vasy--Wunsch~\cite{MVW} or Vasy~\cite{Vasy08}), but in the
present context they can instead be described more simply.  Away from
$\cf$ they are lifts to $\Sbstar M$ of maximally extended light-like
geodesics of $\rho^{2}g$.  At $\cf$, they are concatenations of
bicharacteristics that are continuous as functions to $\Sigmadot$.

In particular, at $\cf$, the broken bicharacteristics are
concatenations of lifts of light-like geodesics entering and exiting
$\cf$; the continuity condition requires that they enter and leave
``at the same time'' (i.e., with the same $\rho$ or $t$ coordinate
along $I$) and with the same ``time momentum'' (i.e., the same value
of $\taub$).  More precisely, straightforward ODE analysis shows
that in the edge cotangent bundle, null bicharacteristics enter
$\Sestar_{\cf}M$ with coordinates
\begin{equation*}
  (\rho, 0, z_{0}, \taue, \xie, 0),
\end{equation*}
with $\taue^{2} = \xie^{2}$.  They then leave $\Sestar_{\cf}M$ from
the point
\begin{equation*}
  (\rho, 0, z_{1}, \taue, -\xie, 0),
\end{equation*}
where $z_{1}$ is a possibly different point in $Z$.\footnote{In other
  words, the direction in which the bicharacteristic leaves the cone
  point need not have any relation to the direction in which it
  entered.  In the parlance of Melrose--Wunsch~\cite{MW}, these are
  the ``diffractive'' bicharacteristics.}  The main result of
Section~\ref{sec:prop-sing-bulk} below is to show that singularities
of $L$ propagate only along these broken bicharacteristics.

\subsection{Broken bicharacteristics for the operator $\Ps$}
\label{sec:operator-ps}

The classical propagation for the operator $\Ps$ near $\pd X = \{
x=0\}$ is simpler to describe as $\Ps$ is classically elliptic there:
there is no propagation.  On the other hand, the related semiclassical
operator
\begin{equation*}
  \Ph = h^{2} \Ps, \quad h = \abs{\sigma}^{-1},
\end{equation*}
is not semiclassically elliptic.

We consider the characteristic set $\Sigmah$ of the operator $\Ph$
near $x=0$.  The principal symbol of $\Ph$ in this region is
\begin{equation*}
  \sigmabh (\Ph) = \left( \lambda + \xib\right)^{2} -
  \frac{\xib^{2}}{x^{2}} - \frac{1}{x^{2}}\abs{\zetab}^{2}
\end{equation*}
where $\lambda = \sigma / \abs{\sigma} = \pm 1 + \bo (h)$.  Its
Hamilton vector field is
\begin{equation*}
  \frac{2}{x^{2}}\left( (x^{2}\xib + x^{2}\lambda - \xib)x\pd[x]
  - (\xib^{2} + \abs{\zetab}^{2})\pd[\xi]\right) - \frac{1}{x^{2}}H_{\abs{\zetab}^{2}},
\end{equation*}
where $H_{\abs{\zetab}^{2}}$ is the Hamilton vector field of the
metric function $\met^{-1}$.  Within the characteristic set of $\Ph$,
the only trajectories reaching $x=0$ reach points of the form
\begin{equation*}
  (x=0, z\in Z, \xib=0, \zetab=0),
\end{equation*}
i.e., the analogue of the compressed characteristic set for the
semiclassical operator is the zero section over the boundary.

An analogous construction to the one described for the operator $L$
shows that over $\pd X = \mf \cap \cf$, we have
\begin{equation*}
  \Sigmah\rvert_{\pd X} = \left\{ (x=0, z, \xib=0, \zetab = 0) \mid z
    \in Z \right\},
\end{equation*}
and that the broken bicharacteristics for the Hamilton flow of the
principal symbol of $\Ph$ must enter and leave through a point of this
form with possibly different $z$ values.  

\subsection{The radial sets}
\label{sec:radial-sets}

We finally describe the \emph{radial sets} for the Hamilton flow
associated to the operators $L$ and $\Ph$; these are the sets where
the Hamilton vector field of the principal symbol is a multiple of the
radial vector field $\xib \pd[\xib] + \zetab \pd[\zetab]$.  In both
cases, the radial sets are identical to those described in prior
work~\cite{BVW1}; we include them here for the purpose of
completeness but refer to that work for their characterization.

The radial sets for $\Ps$ and $\Ph$ occur at $N^{*}S_{\pm}$; boundary of the
fiber-compactification acts as a source or sink for the global flow.
We define $\Lambda^{\pm}$ to be the fiber-infinity boundary of
$N^{*}S_{\pm}$ in $\Tbstarbar X$; in
our analysis below we aim to propagate regularity from the radial set
living over $S_{-}$ to the one living over $S_{+}$.

The radial points $\mathcal{R}^{\pm}$ of $L$ also lie over $S_{\pm}$; in terms of
coordinates $(\rho, x, z, \taub, \xib, \zetab)$ in region III, their
image $\pd \mathcal{R}^{\pm} \subset \Tbstarbar M$ in the cosphere bundle is given by
\begin{equation*}
  \left\{ (\rho = 0, x = 1, z, \taub = 0, \xib, \zetab = 0) \mid z \in
    Z, \xib = \pm 1\right\}.
\end{equation*}

\section{Pseudodifferential operators}
\label{sec:pseud-oper}

The main results of this paper all rely on the interaction between
spaces defined using both $\bl$-pseudodifferential operators and conic differential
operators.  While these interactions were key in the analysis of
Melrose--Wunsch in \cite{MW}, their structure was codified and explained
by Vasy in \cite{Vasy08}.  We now describe the spaces of
$\bl$-pseudodifferential operators employed below as well as their
interactions with the generators of the conic differential operators.

\subsection{The homogeneous $\bl$-calculus}
\label{sec:homog-bl-calc}

We now briefly describe the spaces $\Psib^{m}$, $\Psibinf^{m}$, and
$\Psib^{m,\ell}$ of $\bl$-pseudodifferential operators on the bulk
spacetime $M$.  Rather than provide detailed definitions and proofs,
we instead provide a list of their properties and refer the reader to
Melrose's unpublished book~\cite{Melrose:MWC} and Vasy's paper~\cite{Vasy08}
for details.

Our discussion in this section is specialized to a neighborhood of
$\mf \cap \cf$ (region III) in $M$; the relevant results in region I
can be quoted, while the results in region II can be recovered by
assuming that $x$ is bounded away from $0$.

The space of $\bl$-pseudodifferential operators $\Psib^{*}(M)$ is the
``quantization'' of the Lie algebra of vector fields tangent to the
boundary of $M$ and formally consists of operators of the form
\begin{equation*}
  b\left(\rho, x, z, \rho D_{\rho} , x D_{x}, D_{z}\right),
\end{equation*}
where $b$ is a classical symbol (i.e., it is smooth on $\Tbstar M$ and
has a complete asymptotic expansion at fiber infinity).  In terms of
coordinates $(\rho, x, z)$ near the corner $\mf \cap \cf$, we may
write an explicit quantization of the symbol $b$ by
\begin{align*}
  \Op (b) u (\rho, x, z) &= \frac{1}{(2\pi)^{n+1}} \int\int e^{i(\rho
                           - \rho ' )\taub + i (x-x')\xib + i
                           (z-z')\cdot \zetab} \phi \left( \frac{\rho
                           - \rho'}{\rho}\right) \phi
                           \left(\frac{x-x'}{x}\right) \psi (z) \\
  &\quad\cdot b(\rho, x, z, \rho \taub, x \xib, \zetab)u(\rho', x', z')\,d\taub\,d\xib\,d\zetab\,d\rho'\,dx'\,dz',
\end{align*}
where $\phi \in C^{\infty}_{c}((-1/2,1/2))$ is identically $1$ near
$0$, $\psi$ localizes to a region of $Z$ where the local coordinate
$z$ is valid, and the integrals in $\rho'$ and $x'$ are over
$[0,\infty)$.

We further define the multi-filtered algebra $\Psib^{m,\ell}(M) =
\rho^{-\ell}\Psib^{m}(M)$.  The index $\ell$ refers only to the
filtration in $\rho$; we do not explicitly rely on a filtration in $x$
later in the text.

Our regularization arguments in Section~\ref{sec:near-singular-points}
rely (in a similar way to those of Melrose--Vasy--Wunsch~\cite{MVW})
on a slightly larger algebra we call $\Psibinf^{*}(M)$.  It is defined
in the same way but with symbols satisfying Kohn--Nirenberg estimates
(rather than having complete asymptotic expansions).

The algebra $\Psib^{m,\ell}(M)$ satisfies the following properties:
\begin{enumerate}[i.]
\item The principal symbol of a $\bl$-differential operator, defined
  by
  \begin{equation*}
    \sigma_{\bl,m, \ell} \left( \rho^{-\ell}\sum_{j + k +
        \abs{\alpha}\leq m}a_{jk\alpha}(\rho D_{\rho})^{j}(xD_{x})^{k}D_{z}^{\alpha}\right) = \rho^{-\ell}\sum_{j+k+\abs{\alpha}=m}a_{jk\alpha}\taub^{j}\xib^{k}\zeta^{\alpha},
  \end{equation*}
  extends continuously to give a map
  \begin{equation*}
    \sigma_{\bl,m,\ell} : \Psib^{m,\ell}(M) \to
    \rho^{-\infty}C^{\infty}(\Sbstar M).
  \end{equation*}
  The principal symbol map is multiplicative, i.e., $\sigma (AB) =
  \sigma (A) \sigma(B)$.

  In the case of $\Psibinf^{m}(M)$, the principal symbol instead takes
  values in the quotient of the symbol spaces
  \begin{equation*}
    S^{m}(\Tbstar M)/S^{m-1}(\Tbstar M),
  \end{equation*}
  which in the case of classical symbols can be identified with
  $C^{\infty}(\Sbstar M)$.
  
  The principal symbol captures the top order behavior (in $m$) of
  elements of $\Psib^{m,\ell}(M)$.  In other words, the following
  sequence is exact:
  \begin{equation*}
    0 \to \Psib^{m-1,\ell}(M) \to \Psib^{m,\ell}(M) \to
    \rho^{-\ell}C^{\infty}(\Sbstar M) \to 0.
  \end{equation*}
  (In the case of $\Psibinf^{m}$, the symbol space must be replaced by
  the quotient $S^{m}/S^{m-1}$.)
  
\item There is a (non-canonical) quantization map $\Op :
  \rho^{-\ell}S^{m}(\Tbstar M) \to \Psib^{m,\ell}(M)$ so that
  \begin{equation*}
    \sigma_{\bl,m,\ell}(\Op(a)) = a
  \end{equation*}
  as an element of $\rho ^{-\ell}S^{m}(\Tbstar M)/ \rho
  ^{-\ell}S^{m-1}(\Tbstar M)$.  
\item The algebras $\Psib^{m,\ell}(M)$ and $\Psibinf^{m}(M)$ are
  closed under adjoints, and
  \begin{equation*}
    \sigma (A^{*}) = \overline{\sigma (A)}.
  \end{equation*}
\item If $A \in \Psib^{m,\ell}(M)$ and $B\in \Psib^{m',\ell'}(M)$,
  then $[A, B] = AB-BA \in \Psib^{m+m'-1, \ell + \ell'}(M)$, and
  \begin{equation*}
    \sigma_{\bl, m + m' -1 ,\ell + \ell'} \left( i [A,B]\right) =
    \left\{ \sigma (A), \sigma(B)\right\},
  \end{equation*}
  where the right hand side denotes the Poisson bracket induced by the
  symplectic structure on $\Tbstar M$ as in Section~\ref{sec:basics-bl-geometry}.
\item Elements of $\Psib^{0}(M)$ are bounded on $L^{2}$.  In
  particular, given $A \in \Psib^{0}(M)$, there is an $A'\in
  \Psib^{-1}(M)$ so that
  \begin{equation*}
    \norm{Au}_{L^{2}} \leq 2 \sup \abs{\sigma(A)}\norm{u}_{L^{2}} + \norm{A'u}_{L^{2}}.
  \end{equation*}
  
\item If $A \in \Psib^{m,\ell}(M)$ (or $\Psibinf^{m}(M)$), the
  \emph{microsupport} (or operator wavefront set)
  $\WFb'(A) \subset \Sbstar M$ of $A$ is the set of points and
  directions in which the total symbol of $A$ fails to be rapidly
  decaying, and obeys the usual microlocality property:
  \begin{equation*}
    \WFb'(AB) \subset \WFb'(A) \cap \WFb'(B).
  \end{equation*}
\end{enumerate}

The analysis below requires commuting $\bl$-pseudodifferential
operators with the components $D_{x}, \frac{1}{x}D_{z}$, and
$\frac{1}{x}$ of the operators on the cone.  As commutators with
$\frac{1}{x}D_{z}$ are not necessarily lower order, we are careful to
select commutants that commute with derivatives in $z$ to top order.
In other words, we require the notion of a \emph{basic operator}
introduced by Melrose--Vasy--Wunsch~\cite[Section 9]{MVW}.
\begin{definition}
  We say a symbol $a \in C^{\infty}(\Tbstar M)$ is basic if
  $\pd[z]a = 0$ at $\{x =0, \xib=0, \zetab=0\}$.  The quantization of
  such a symbol is called a basic operator.
\end{definition}

We now recall from Melrose--Vasy--Wunsch~\cite[Lemma 8.6]{MVW} how
the $\bl$-calculus interacts with $\frac{1}{x}, D_{x},$ and
$\frac{1}{x}D_{z_{j}}$.
\begin{lemma}
  \label{lem:commuting-through}
  If $A\in \Psib^{m}(M)$, then there are $B\in \Psib^{m}(M)$ and $C\in
  \Psib^{m-1}(M)$ depending continuously on $A$ so that
  \begin{equation*}
    i \left[ D_{x}, A\right] = B + C D_{x},
  \end{equation*}
  with $\sigma(B) = \pd[x]\sigma(A)$ and $\sigma(C) =
  \pd[\xib]\sigma(A)$.

  Similarly, there are $C_{L} , C_{R}\in \Psib^{m-1}(M)$ with
  $\sigma(C) = \pd[\xib]\sigma(A)$ so that
  \begin{equation*}
    i \left[ \frac{1}{x}, A \right] = C_{L}\frac{1}{x} = \frac{1}{x} C_{R}.
  \end{equation*}

  If, in addition, $A$ is a basic operator, then
  \begin{equation*}
    i \left[ \frac{1}{x}D_{z_{j}} , A\right] = B_{j} +C_{j}D_{x} +
    \sum_{k}E_{jk}\frac{1}{x}D_{z_{k}} + \frac{1}{x}F_{j},
  \end{equation*}
  with $B_{j}\in \Psib^{m}(M)$, $C_{j}, E_{jk}, F_{j} \in
  \Psib^{m-1}(M)$, and
  \begin{equation*}
    \pd[z_{j}]\sigma(A) + \zetab_{j}\pd[\xib]\sigma(A) = x
    \sigma(B_{j}) + \xib \sigma (C_{j}) + \sum_{k}\zetab_{k}\sigma (E_{jk}).
  \end{equation*}
\end{lemma}

As in the work of Vasy~\cite{Vasy08} and
Melrose--Vasy--Wunsch~\cite{MVW}, we define
\begin{equation*}
  x^{-k}\Diffe^{k}\Psib^{m}\subset x^{-k}\Psib^{k+m}
\end{equation*}
to be the span of the products $QA$ with $Q \in x^{-k}\Diffe^{k}$ and
$A \in \Psib^{m}$.  By Lemma~\ref{lem:commuting-through}, it is also
generated by the products $AQ$ and so the union
\begin{equation*}
  \bigcup_{k,m}x^{-k}\Diffe^{k}\Psib^{m}
\end{equation*}
is a bigraded ring closed under adjoints.  This observation allows us
to freely commute elements of $x^{-k}\Diffe^{k}$ through
$\bl$-pseudodifferential operators at the cost of lower order terms.

\subsection{The semiclassical $\bl$-calculus}
\label{sec:semicl-bl-calc}

On the boundary hypersurface $X = \mf$, we further employ the
$\bl$-calculus as well as its semiclassical variant $\Psibh^{*}$.  In
this section we briefly describe properties satisfied by the
semiclassical $\bl$-calculus $\Psibh^{*}(X)$.  We refer the reader to
Gannot--Wunsch~\cite[Section 3]{GW} for details.  We remind the reader
that $X$ is an $n$-dimensional compact manifold with
boundary.\footnote{The shift in dimension arises because we employ the
  semiclassical calculus only on $X= \mf$ rather than the bulk $M$.}

We can obtain an explicit quantization procedure on $X$ near its boundary
in terms of coordinates $(x,z)$ by fixing $\phi \in
C^{\infty}_{c}((-1/2,1/2))$ so that $\phi (s) \equiv 1$ near $s=0$ and
$\psi\in C^{\infty}_{c}$ localizing to a fixed coordinate chart in $z$.
Given a semiclassical symbol $a\in S^{m}_{h}(\Tbstar X)$, define
$\Opbh(a)\in \Psibh^{m}(X)$ by
\begin{align*}
 &  \Opbh (a)u(x,z) =  \\
& \hspace{.5cm}  \frac{1}{(2\pi h)^{n}}\int \int e^{i\left(
      (x-x')\xib + (z-z')\cdot \zetab\right)} \phi\left(
    \frac{x-x'}{x}\right) \psi (z) a(x,z,x\xib, \zetab)u(x',z') \,d\xib \,d\zetab\,dx'\,dz'.
\end{align*}

As in the homogeneous setting, the space of semiclassical
$\bl$-pseudodifferential operators on $X$ satisfies the following
properties:
\begin{enumerate}[i.]
\item There is a principal symbol map $\sigmabh : \Psibh^{m}(X) \to
  S^{m}(\Tbstar X) / hS^{m-1}(\Tbstar X)$ so that the sequence
  \begin{equation*}
    0 \to h \Psibh^{m-1}(X) \to \Psibh^{m}(X) \to S^{m}(\Tbstar X) / h
    S^{m-1}(\Tbstar X) \to 0
  \end{equation*}
  is exact.  This map is multiplicative.
\item There is a (non-canonical) quantization map $\Opbh :
  S^{m}(\Tbstar X) \to \Psibh^{m}(X)$ so that if $a \in S^{m}(\Tbstar
  X)$, then
  \begin{equation*}
    \sigmabh (\Opbh (a)) = a
  \end{equation*}
  as an element of $S^{m}(\Tbstar X) / h S^{m-1}(\Tbstar X)$.
\item The algebra $\Psibh^{*}(X)$ is closed under adjoints and
  \begin{equation*}
  \sigmabh (A^{*}) = \overline{\sigmabh(A)}.
\end{equation*}
\item If $A \in \Psibh^{m}(X)$ and $B \in \Psibh^{m'}(X)$, then
  $[A,B]\in h\Psibh^{m+m'-1}(X)$ and has principal symbol
  \begin{equation*}
    \sigmabh \left( \frac{i}{h}\left[A,B\right]\right) = \left\{
      \sigmabh(A), \sigmabh(B)\right\},
  \end{equation*}
  where the Poisson bracket is taken with respect to the symplectic
  structure on $\Tbstar X$.
\item Each $A \in \Psibh^{0}(X)$ extends to a bounded operator on
  $L^{2}$ and there exists $A' \in \Psibh^{-\infty}(X)$ so that
  \begin{equation*}
    \norm{Au}_{L^{2}} \leq 2 \sup \abs{\sigmabh(A)}\norm{u}_{L^{2}} +
    \bo (h^{\infty}) \norm{A'u}_{L^{2}}.
  \end{equation*}
\item If $A \in \Psibh^{*}(X)$, the microsupport (or operator
  wavefront set) $\WFbh'(A) \subset \Tbstar X$ is the set of points in
  the $\bl$-cotangent bundle at which $A$ fails to lie in
  $h^{\infty}\Psibh^{-\infty}$.  It obeys the standard microlocality property:
  \begin{equation*}
    \WFbh'(AB) \subset \WFbh'(A) \cap \WFbh'(B).
  \end{equation*}
\end{enumerate}

As in the homogeneous setting, we say that a \emph{basic operator} is
the quantization of a symbol $a$ with $\pd[z] a = 0$ at $\{ x = 0,
\xib = 0, \zetab = 0\}$.  We also require the semiclassical analogue
of Lemma~\ref{lem:commuting-through}, with proof essentially identical
to the one in the homogeneous setting.

\begin{lemma}[{cf.~\cite[Lemma 8.6]{MVW} and~\cite[Lemma 3.6]{GW}}]
  \label{lem:commuting-through-scl}
  If $A \in \Psibh^{m}(X)$, there are $B \in \Psibh^{m}(X)$ and $C\in
  \Psibh^{m-1}(X)$ so that
  \begin{equation*}
    \frac{i}{h}\left[ hD_{x}, A\right] = B + C (hD_{x}),
  \end{equation*}
  with $\sigmabh (B) = \pd[x]\sigmabh (A)$ and $\sigmabh (C) =
  \pd[\xib]\sigmabh(A)$.

  Moreover, there are $C_{L}, C_{R} \in \Psibh^{m-1}(X)$ with
  $\sigmabh(C_{\bullet}) = \pd[\xib]\sigmabh (A)$ and
  \begin{equation*}
    i \left[ \frac{1}{x} , A\right] = \frac{h}{x}C_{R} = C_{L} \frac{h}{x}.
  \end{equation*}

  If, in addition, $A$ is a basic operator, then
  \begin{equation*}
    \frac{i}{h} \left[ \frac{h}{x}D_{z_{j}}, A \right] = B_{j} +
    C_{j}(hD_{x}) + \sum_{k}E_{jk}\frac{h}{x} D_{z_{k}} + \frac{h}{x}F_{j},
  \end{equation*}
  where $B_{j}\in \Psibh^{m}(X)$, $C_{j}, E_{jk}, F_{j}\in
  \Psibh^{m-1}(X)$, and
  \begin{equation*}
    \pd[z_{j}]\sigmabh(A) + \zetab_{j} \pd[\xib]\sigmabh(A) = x
    \sigmabh(B_{j}) + \xib \sigmabh(C_{j}) + \sum_{k}\zetab_{k}\sigmabh(E_{jk}).
  \end{equation*}
\end{lemma}

Just as in the homogeneous setting,
Lemma~\ref{lem:commuting-through-scl} allows us to freely commute
factors of $hD_{x}$, $\frac{h}{x}D_{z_{j}}$, and $\frac{1}{x}$ through
semiclassical $\bl$-pseudodifferential operators at the cost of lower
order terms.

\section{Function spaces}
\label{sec:funct-spac-wavefr}

As described above, our analysis is based on mixed differential-pseudodifferential
structures on both $M$ and $\mf$.  The associated analytic objects we
employ are therefore adapted to the Friedrichs form domain of the
conic Laplacian.

We denote by $\dom$ the Friedrichs form domain of the Laplacian on the
cone $C(Z)$, i.e., the domain of $\lap^{1/2}$, where $\lap$ is the
Friedrichs extension of the Laplacian.  It is equipped with a natural
norm given by
\begin{equation*}
  \norm{u}_{\dom}^{2} = \norm{u}^{2} + \langle \lap u , u \rangle,
\end{equation*}
where the norm and inner product are taken with respect to the $L^{2}$
space induced by the conic metric on $C(Z)$.  Writing the Laplacian in
coordinates, the norm on $\dom$ is controlled by
\begin{equation*}
  \norm{v} + \norm{\pd[r] v} + \norm{r^{-1}\grad_{z}v},
\end{equation*}
where the pointwise magnitude of the last term is measured by the
metric $k$ on the cross-section.  

Just as in Euclidean space in three dimensions and higher, $\dom$
enjoys an analogue of the Hardy inequality:
\begin{lemma}[{\cite[Lemma 5.2]{MVW}}]
  \label{lem:base-domain}
  If $\dim Z > 1$, then there is some $C$ so that for all $v \in
  C^{\infty}_{c}((0,\infty)\times Z)$,
  \begin{equation*}
    \norm{r^{-1}v}^{2} +\norm{v}^{2}+ \norm{r^{-1}\grad_{z}v}^{2} +
    \norm{\pd[r]v}^{2} \leq C \norm{v}_{\dom}^{2}.
  \end{equation*}
\end{lemma}
We often appeal to Lemma~\ref{lem:base-domain} and its analogues in
order to estimate terms of the form $r^{-1}v$ arising in commutator
estimates; the lemma asserts that they can be controlled by one
``conic derivative''.  Just as the Hardy inequality fails in dimension
two, Lemma~\ref{lem:base-domain} is false when $\dim Z = 1$.  On the
other hand, when $\dim Z = 1$, the manifold $Z$ must be a circle, hence
modifications similar to those used by
Melrose--Vasy--Wunsch~\cite[Section 10]{MVW} allow us to recover the
propagation results of Sections~\ref{sec:prop-sing-bulk}
and~\ref{sec:prop-sing} in this case.  For the purpose of exposition,
we omit these arguments in this paper.

\subsection{$\bl$-Sobolev spaces on the bulk $M$}
\label{sec:bl-sobolev-spaces}

Although it is standard in the $\bl$-calculus literature to define
$\bl$-Sobolev spaces with respect to a fixed $\bl$-density, the proofs
in Sections~\ref{sec:prop-sing-bulk} and~\ref{sec:prop-sing} below
more naturally employ a cone-type density, i.e., a rescaling of the
density reflecting the conic structure of the problems.

In other words, on $M$ we consider the density associated to the
Lorentzian metric $\rho^{2}g$, which in local coordinates has the form
\begin{equation*}
  \frac{x^{n-1}\sqrt{\met}}{\rho}\,d\rho\,dx\,dz.
\end{equation*}
All $L^{2}$ norms on $M$ are taken with respect to this density.

We let $\Hb^{m}(M)$ denote the $\bl$-Sobolev space of order $m$
relative to the function space $L^{2}(M)$ and the algebras
$\Diffb^{m}(M)$ and $\Psib^{m}(M)$.  In particular, for $m \geq 0$, if
$A \in \Psib^{m}(M)$ is a fixed invertible elliptic operator, then
$u \in \Hb^{m}(M)$ if and only if $u, Au \in L^{2}(M)$.\footnote{If
  $m$ is a positive integer, $\Hb^{m}$ can be characterized in terms
  of $\Diffb^{m}(M)$.  A characterization for other values of $m$ then
  follows by interpolation and duality.}  For $m< 0$, the space
$\Hb^{m}(M)$ is defined as the dual space of $\Hb^{-m}(M)$ with
respect to the $L^{2}(M)$ pairing.  We further require an additional
filtration of the Sobolev spaces.  For $\ell \in \reals$, we set
$\Hb^{m,\ell}(M) = \rho^{\ell}\Hb^{m}(M)$ as defined in for instance Section $2$ of \cite{MW}.

In an abuse of notation, we use $\domst$ to denote a differential
Sobolev space of order $1$ on the spacetime $M$:
\begin{definition}
  We let $\domst$ denote the set of functions $u \in \Hb^{1}(M)$ for
  which the norms $\norm{\pd[x]u}$ and $\norm{x^{-1}\grad_{z}u}$ are
  both finite.
\end{definition}
Just as it is well-known (see, e.g., \cite[Section
3]{MW}\footnote{Because we adopt the convention that $L^{2}$ and
  $\bl$-Sobolev are measured with respect to the metric density on
  $C(Z)$ rather than a $\bl$-density, this identification of $\dom$
  with a $\bl$-Sobolev space differs from the one in that paper by a
  factor of $r^{-n/2}$.}) that $\dom = r \Hb^{1}(C(Z))$, we could
instead define $\domst$ as a type of weighted $\bl$-Sobolev space with only
partial regularity.  It is convenient, however, for our purposes, to
ensure that $\rho\pd[\rho]$ and $\pd[x]$ are on nearly equal footing.

Away from $\cf$, $\domst$ is a standard $\bl$-Sobolev space (defined
with respect to the density above).  Near $\cf$, it inherits the norm
\begin{equation*}
  \norm{u}_{\domst}^{2} = \norm{u}^{2} + \norm{\rho\pd[\rho]u}^{2}+ \norm{\pd[x] u}^{2} + \norm{x^{-1}\grad_{z}u}^{2},
\end{equation*}
and is closed with respect to this norm. Just as in
Lemma~\ref{lem:base-domain}, we have the additional Hardy-type
inequality
\begin{equation*}
  \norm{x^{-1}u}\leq C \norm{u}_{\domst}.
\end{equation*}

As solutions of the wave equation are not typically $L^{2}$ in time,
we require a weighted variant of $\domst$: for $\ell \in \reals$, we
let $\rho^{-\ell}\domst$ denote the space of those $u$ for which
$\rho^{\ell} u \in \domst$.  The $\rho^{-\ell}\domst$ norm of a
distribution $u$ is the $\domst$ norm of $\rho^{\ell}u$.

Integrating energy estimates shows that solutions of the wave equation
with compactly supported finite energy initial data\footnote{We state
  and prove the result for the forward problem with smooth compactly
  supported initial data, but an inspection of the proof reveals that
  it needs only finite energy and compact support.} lie in
$\rho^{-\ell}\domst$ for some $\ell$.

\begin{lemma}
  \label{lem:there-are-admissible-solutions}
  If $u$ is the forward solution of $Lu = f$, where $f \in
  \CI_{c}(M^{\circ})$, then there is some $\ell \in \reals$ so that $u
  \in \rho^{-\ell}\domst$.
\end{lemma}

\begin{proof}
  The lemma follows by energy bounds and changing variables.  Indeed,
  for each fixed $t$, standard energy estimates show that
  \begin{equation*}
    \int_{C(Z)} \left(\abs{\pd[t]u}^{2} + \abs{\pd[r]u}^{2} +
      \abs{\frac{1}{r} \grad_{z}u}^{2} \right) \sqrt{\met} r^{n-1}\,dr\,dz
    < C(f),
  \end{equation*}
  and so the Hardy-type inequality also shows that
  \begin{equation*}
    \int_{C(Z)}\abs{r^{-1}u}^{2}\sqrt{\met}r^{n-1}\,dr\,dz <C(f)
  \end{equation*}
  is uniformly bounded.  Integrating these estimates in time shows
  that for any $\alpha > 1/2$, we have
  \begin{equation*}
    \int_{t}\int_{C(Z)}\left( \abs{r^{-1}u}^{2}+ \abs{\pd[t]u}^{2} +
      \abs{\pd[r]u}^{2} + \abs{\frac{1}{r}\grad_{z}u}^{2}\right)
    \sqrt{\met}r^{n-1}\ang{t}^{-2\alpha}\,dr\,dz\,dt  <C_{\alpha}(f),
  \end{equation*}
   where this new constant differs from the previous one by a factor of
  $\int_{\reals}\ang{t}^{-2\alpha}\,dt$. 

  By the finite speed of propagation and possibly translating the
  coordinate system in $t$, it suffices to estimate the
  $\rho^{-\ell}\domst$ norm of $u$ in the region where $r \leq At$ for
  some $A > 1$.  We may use the boundary defining function
  $\rho = 1/t$ in this region and $x = r/t$ as a replacement for the
  radial coordinate and then the region corresponds to $x \leq A$.  We
  then aim to show that there is some $\ell$ for which
  \begin{equation*}
    \int_{0}^{\rho_{0}}\int_{0}^{A} \int_{Z} \left( \abs{u}^{2} + \abs{\rho
        \pd[\rho]u}^{2} + \abs{\pd[x] u}^{2} +
      \abs{x^{-1}\grad_{z}u}^{2}\right) \sqrt{\met} \,dz \,x^{n-1}\,dx \,
    \rho^{2\ell}\,\frac{d\rho}{\rho} < \infty.
  \end{equation*}
  Changing coordinates back to $(t,r,z)$, this is equivalent to
  estimating
  \begin{equation*}
    \int_{t_{0}^{-1}}^{\infty}\int_{0}^{At}\int_{Z}\left( \abs{u}^{2}
      + \abs{t\pd[t]u + r \pd[r]u}^{2} + \abs{t\pd[r]u}^{2} +
      \abs{\frac{t}{r}\grad_{z}u}^{2}\right) \sqrt{\met}\,dz
    r^{n-1}\,dr\, t^{-n+1-2\ell}\,\frac{dt}{t}.
  \end{equation*}
  As $r \leq At$ in this region, this integral is bounded by $C_{\ell
    + \frac{n}{2} - 1}(f)$, provided that $\ell -1 + n/2 > 1/2$.
\end{proof}

The main use of the space $\domst$ is to act as the base level against
which we measure regularity of distributions on $M$.  To that end, we
let $\Hbdst^{1}(M) = \domst$ and define, for $m \geq 1$ and
$\ell \in \reals$, the finite order conormal spaces
$\Hbdst^{m,\ell}(M)$:
\begin{definition}
  Let $A \in \Psib^{m-1}(M)$ be an invertible basic (in the sense of
  Section~\ref{sec:pseud-oper}) elliptic operator.  For $m\geq 1$, the space
  $\Hbdst^{m,\ell}(M)$ consists of those $u \in \rho^{\ell}\domst$ for
  which $Au \in \rho^{\ell}\domst$.
\end{definition}
In other words, $\Hbdst^{m,\ell}(M)$ consists of those distributions
conormal to $\mf$ and $\cf$ of finite order $m-1$ relative to
$\rho^{\ell}\domst$.  Away from $\cf$, they agree with the weighted
$\bl$-Sobolev spaces $\rho^{\ell}\Hb^{m}(M)$ and indeed we have the
inclusion $\Hbdst ^{m}\hookrightarrow \Hb^{m}(M)$.

The following lemma shows that these spaces do not depend on the
choice of basic $A$ (as in the work of Vasy~\cite[Remark
3.6]{Vasy08}), as basic operators of order $0$ preserve $\domst$.
\begin{lemma}
  \label{lem:psib-acts-on-st-domain}
  If $A \in \Psib^{0}(M)$ is a basic operator, then
  \begin{align*}
    A:\rho^{\ell}\domst \to \rho^{\ell}\domst, \quad A : \rho^{\ell}\domst' \to \rho^{\ell}\domst'
  \end{align*}
  are bounded.
\end{lemma}

\begin{proof}
  As conjugation by $\rho^{\ell}$ yields another basic element of
  $\Psib^{0}(M)$, we must prove the lemma only for $\ell = 0$.
  
  The result follows from the commutator expressions of
  Lemma~\ref{lem:commuting-through-scl}.  Indeed, to estimate
  $\norm{Au}_{\domst}$, it suffices to estimate the quantities
  \begin{equation*}
    \norm{\rho \pd[\rho] Au}, \norm{\pd[x] Au},
    \norm{\frac{1}{x}\grad_{z}Au}, \norm{Au},
  \end{equation*}
  where all norms taken are with respect to $L^{2}$.

  We show explicitly this bound only for the term
  $\frac{1}{x}\grad_{z}Au$; the $\pd[x]Au$ term is treated similarly
  while the other two terms amount to the boundedness of $\Psib$ on
  $\bl$-Sobolev spaces.  Appealing to
  Lemma~\ref{lem:commuting-through-scl}, we write
  \begin{equation*}
    \frac{1}{x}\pd[z_{j}] Au = A \frac{1}{x}\pd[z_{j}]u + B_{j}u +
    C_{j} \pd[x]u + \sum_{k}E_{jk}\frac{1}{x}D_{z_{k}}u + F_{j}\frac{1}{x}u ,
  \end{equation*}
  where $B_{j}\in \Psib^{0}$, and $C_{j}, E_{jk}, F_{j} \in
  \Psib^{-1}$.  As elements of $\Psib^{s}$ are bounded on $L^{2}$ for
  $s\leq 0$, we may then estimate
  \begin{equation*}
    \norm{\frac{1}{x}\pd[z_{j}]Au}_{L^{2}} \leq C \left(
    \sum_{k}\norm{\frac{1}{x}\pd[z_{k}]u}_{L^{2}} +
    \norm{\pd[x]u}_{L^{2}} + \norm{\frac{1}{x}u} +  \norm{u}_{L^{2}}
  \right) \leq C \norm{u}_{\domst}.
  \end{equation*}
\end{proof}

We finally describe a microlocal characterization of regularity, the
wavefront set.
\begin{definition}
  Let $u \in \Hbdst^{s,\ell}$ for some $s\geq 0$ and $\ell\in \reals$ and
  suppose that $m\geq 0$.  We say $q \in \Tbstar M\setminus 0$ is not
  in $\WFbdst^{m,\ell}(u)$ if there is some $A \in \Psib^{m,\ell}(M)$
  elliptic at $q$ so that $Au \in \domst$.

  For $m = \infty$, $q$ is not in $\WFbdst^{\infty,\ell}(u)$ if there
  is some $A\in \Psib^{m,\ell}(M)$ elliptic at $q$ with $Au \in \Hbdst^{\infty,\ell}$.
\end{definition}

Note that if $\WFbdst^{\infty,\ell}(u) = \emptyset$, then $u$ is
fully conormal to $\mf$ and $\cf$ relative to the space $\rho^{\ell}\domst$.

\subsection{Variable-order Sobolev spaces on the boundary $\mf$}
\label{sec:vari-order-sobol-1}

We now turn our attention to the function spaces on the boundary
$\mf$.  We fix a density on $\mf$ against which we integrate
functions; away from the boundary of $\mf$ we ask only that it be
smooth and nondegenerate, while at the boundary $\mf \cap \cf$ of
$\mf$, we demand that it take the following form in local coordinates
$(x, z)$:
\begin{equation*}
  x^{n-1}\sqrt{\met}\,dx\,dz.
\end{equation*}

Near $\mf\cap \cf$ (i.e., near the boundary of
$X=\mf$), the operator $\Ps$ is a conjugate of the Laplacian on a
hyperbolic cone (see Section~\ref{sec:proof-theor-refth}).  In fact,
near the boundary $\Ps$ differs from the Laplacian on $C(Z)$ by an
element of $\Diffb^{2}(\mf)$.  Lemma~\ref{lem:base-domain} applies on
$\mf$ as well, motivating the following abuse of notation:
\begin{definition}
  We let $\dom$ denote the space of functions on $\mf$ that:
  \begin{enumerate}
  \item lie in $H^{1}$ away from $\mf \cap \cf$, and 
  \item lie in the Friedrichs form domain of $\lap_{C(Z)}$ near $\mf
    \cap \cf$.  
  \end{enumerate}

  We let $\dom'$ denote the dual of $\dom$ with respect to the $L^{2}$ pairing.
\end{definition}

A more patently invariant way to define $\dom$ involves fixing an
invertible elliptic operator (in, e.g., Hintz's semiclassical cone
calculus \cite{hintz2020resolvents}) agreeing with $(1 + \lap_{C(Z)})^{1/2}$ near the
boundary.  As our function spaces depend on $\dom$ only near the
boundary, however, we need not take this approach.

As the main propagation result in Section~\ref{sec:prop-sing} is
semiclassical, we introduce a rescaled version of the domain norm,
denoted $\domh$.  For $u$ supported near the boundary of $\mf$, this
norm is given by
\begin{equation*}
  \norm{u}_{\domh}^{2} = \norm{u}^{2} + \norm{h \pd[x] u}^{2} + \norm{\frac{h}{x}\grad_{z}u}^{2}.
\end{equation*}
Lemma~\ref{lem:base-domain} then shows that $\norm{u}_{\domh}$ also
controls $h\norm{x^{-1}u}$.  As above, we use $\domh'$ to denote the
dual of $\domh$.

We point out that the characterization of $\dom$ stemming from
Lemma~\ref{lem:base-domain} shows that the inclusions $\dom
\hookrightarrow L^{2}$ and $L^{2}\hookrightarrow\dom'$ are compact.
This observation is crucial to the Fredholm statement proved in
Section~\ref{sec:fredholm-property}.  When $\dim Z = 1$, the
characterization of the Friedrichs form domain given by
Melrose--Wunsch~\cite[Equation 3.11]{MW} also shows the compactness of
these inclusions.

Just as in the bulk spacetime, pseudodifferential operators of order 0
(and their semiclassical counterparts) preserve these spaces.  The
following lemma is proved in the same way as its classical analogue
(Lemma~\ref{lem:psib-acts-on-st-domain}):
\begin{lemma}
  \label{lem:psib-semiclassical}
  If $A \in \Psibh^{0}(X)$ is a basic operator, then
  \begin{align*}
    A: \domh \to \domh, \quad A : \domh' \to \domh'
  \end{align*}
  are bounded.
\end{lemma}

As we aim to reduce problems on the bulk spacetime $M$ to problems on
its main boundary hypersurface $X = \mf$, we record the following
lemma relating the spaces $\dom$ and $\domst$.  The proof of the lemma
with $\domst$ replaced by a Sobolev space $H^{k}$ is standard; the
proof for $\domst$ proceeds identically.

\begin{lemma}[{cf.~\cite[Lemma 2.3]{BVW1}}]
  Suppose $u \in \rho^{-\ell}\domst$ and that $\chi_{1}, \chi_{2} \in
  C^{\infty}_{c}([0,\infty))$ with $\chi_{2}$ supported in $\{ x <
  1/4\}$.  The Mellin transform (in $\rho$) of
  $\chi_{1}(\rho)\chi_{2}(x)u$ is a holomorphic function for $\Im
  \sigma > \ell$ taking values in $L^{\infty}_{\Im\sigma}L^{2}_{\Re
    \sigma}(\reals ; \dom)$.
\end{lemma}

We now describe the Sobolev spaces on which $\Ps$ is a Fredholm
operator.  As in prior work, these have variable orders;
see~\cite[Appendix A]{BVW1} for details.

We fix a future regularity function $s_{\tow}: \Sbstar \mf \to \reals$
satisfying the following:
\begin{enumerate}
\item $s_{\tow}$ is constant near $\Lambda^{\pm}$ and $s_{\tow} \equiv
  1$ in a neighborhood of the conic singularity $\pd X$,
\item Along the flow in the classical characteristic set of $\Ps$
  (oriented so as to flow from $\Lambda^{-}$ to $\Lambda^{+}$),
  $s_{\tow}$ is monotonically decreasing, and
\item $s_{\tow}$ is less than the threshold exponent at $\Lambda^{+}$
  and greater than the threshold at $\Lambda^{-}$.  
\end{enumerate}
As the classical characteristic set of $\Ps$ lies solely over the
closure of $C_{0}$, the first condition is always compatible with the
second and third.

Note that the thresholds at $\Lambda^{\pm}$ are $\sigma$-dependent, 
so the spaces we consider necessarily depend on which operators in the
family $\Ps$ are under consideration.  Indeed, as in the previous
paper~\cite[Section 5]{BVW1}, the thresholds are given by
\begin{align*}
  \frac{1}{2} &+ \Im \sigma \quad \text{for }\Ps, \\
  \frac{1}{2} &- \Im \sigma \quad \text{for }\Ps^{*}.
\end{align*}

We further define $s_{\tow}^{*} = - s_{\tow} + 1$.  With these
functions in hand, we define (as in~\cite[Appendix A]{BVW1}) the
variable order Sobolev spaces $H^{s_{\tow}}$ and $H^{s_{\tow}^{*}}$
away from the conic singularity $\partial \mf$.  Recall that standard
elliptic regularity estimates still hold in these spaces; hyperbolic
propagation estimates also remain valid provided that the order
function is decreasing along the flow.

We now fix a partition of unity $\phi, 1-\phi \in \CI (\mf)$ so that
$\phi$ is supported near the conic singularity where $s_{\tow} \equiv
1$ and $1-\phi \equiv 0$ in a neighborhood of $\partial \mf$.  We now
define the spaces\footnote{In analogy with the definition of the
  $\domst$-based function spaces on the full spacetime, we could have
  defined the $\cY^{s_{\tow}}$ space more directly using $\dom$.  We
  take the approach above to avoid translating the variable-order
  Sobolev spaces into the $\bl$-setting.}
\begin{align*}
  \cY^{s_{\tow}-1} &= \left\{ u = (1-\phi) u_{1} + \phi u_{2} \mid
                     u_{1} \in H^{s_{\tow}-1}, \ u_{2}\in
                     L^{2}\right\}, \\
  \cY^{s_{\tow}^{*}-1} &= \left\{ u = (1-\phi)u_{1} + \phi u_{2} \mid
                         u_{1} \in H^{s_{\tow}^{*}-1}, \ u_{2} \in
                         \dom '\right\},
\end{align*}
where we have abused notation slightly: the spaces $\cY^{s_{\tow}}$
and $\cY^{s_{\tow}^{*}}$ differ by how they look near the conic
singularity.  As $s_{\tow} = 1$ near the cone points,
$\cY^{s_{\tow}-1}$ agrees with $L^{2}$ there, while $s_{\tow}^{*} = 0$ near
these points, so $\cY^{s_{\tow}^{*}-1}$ is a stand-in for $H^{-1}$
there.  We equip these two spaces with the norms
\begin{align*}
  \norm{u}_{\cY^{s_{\tow}-1}} ^{2} &=
                                     \norm{(1-\phi)u}_{H^{s_{\tow}-1}}^{2}
                                     + \norm{\phi u }_{L^{2}}^{2},\\
  \norm{u}_{\cY^{s_{\tow}^{*}-1}} ^{2} &=
                                         \norm{(1-\phi)u}_{H^{s_{\tow}^{*}-1}}^{2}
                                         + \norm{\phi u}_{\dom'}^{2}.
\end{align*}
The semiclassical versions of these norms are defined by replacing the
Sobolev part of the norm with a semiclassical Sobolev norm and
replacing the $\dom'$ part of the norm with the $\domh'$ norm.

We again rely on the localizer $\phi$ to define the $\cX^{s_{\tow}}$ spaces:
\begin{align*}
  \cX^{s_{\tow}} &= \left\{ u = (1-\phi)u_{1} + \phi u_{2} \mid u_{1}
                   \in H^{s_{\tow}} , \ u_{2} \in \dom, \ \Ps u \in
                   \cY^{s_{\tow}-1}\right\},\\
  \cX^{s_{\tow}^{*}} &= \left\{ u = (1-\phi) u_{1} + \phi u_{2} \mid
                       u_{1} \in H^{s_{\tow}^{*}}, \ u_{2} \in L^{2},
                       \ \Ps u \in \cY^{s_{\tow}^{*}-1}\right\}.
\end{align*}
We have abused notation in the same way as in the definitions of the
$\cY$ spaces.\footnote{Just as we built the Sobolev spaces in the full
  spacetime on top of $\domst$, we have built $\cX^{s_{\tow}}$ on
  $\dom$, $\cY^{s_{\tow}^{*}-1}$ on $\dom'$, and the other two spaces
  on $L^{2}$.}  Observe also that the condition on $\Ps u$ in the
definition of the $\cX$ spaces is independent of $\sigma$ as $\sigma$
only appears in subprincipal (both classically and semiclassically)
terms in $\Ps$.  The norms on the $\cX$ spaces are given by
\begin{align*}
  \norm{u}_{\cX^{s_{\tow}}}^{2} &=
                                  \norm{(1-\phi)u_{1}}_{H^{s_{\tow}}}^{2}
                                  + \norm{\phi u_{2}}_{\dom}^{2} +
                                  \norm{\Ps u}_{\cY^{s_{\tow}-1}}^{2},
  \\
  \norm{u}_{\cX^{s_{\tow}^{*}}} ^{2} &=
                                       \norm{(1-\phi)u_{1}}_{H_{s_{\tow}^{*}}}^{2}
                                       + \norm{\phi u_{2}}_{L^{2}}^{2}
                                       + \norm{\Ps u}_{\cY^{s_{\tow}^{*}-1}}^{2},
\end{align*}
with the semiclassical analogues obtained in the same way as for the
$\cY$ spaces.

One of the main reasons for this setup is that the dual of
$\cY^{s_{\tow}-1}$ consists of those distributions of the form
$(1-\phi)u_{1} + \phi u_{2}$, where $u_{1} \in H^{s_{\tow}^{*}}$ and
$u_{2} \in L^{2}$.  Similarly, the dual of $\cY^{s^{*}-1}$ consists of
those distributions $(1-\phi)u_{1} + \phi u_{2}$ with
$u_{1} \in H^{s_{\tow}}$ and $u_{2} \in \dom$.  Moreover, because the
inclusions $\dom \hookrightarrow L^{2}$ and
$L^{2}\hookrightarrow\dom'$ are compact, the inclusions
$\cX^{s_{\tow}}\hookrightarrow \cY^{s_{\tow}-1}$ and
$\cX^{s_{\tow}^{*}}\hookrightarrow \cY^{s_{\tow}^{*}-1}$ are also
compact.

As our results in Section~\ref{sec:prop-sing} are stated entirely in
terms of estimates, it is unnecessary to define the wavefront set
associated to these spaces.

In Section~\ref{sec:proof-theor-refth} below, we also use
variable-order $\bl$-Sobolev spaces $\Hb^{s_{\tow}}$ \emph{not} based on $\dom$.
As $s_{\tow}$ is constant near $\pd X$, these spaces can be defined in
the standard way (see, e.g., \cite[Appendix A]{BVW1}).  We note that,
with our definitions of $s_{\tow}$ and $s_{\tow}^{*}$, we have the inclusions
\begin{align*}
  \cX^{s_{\tow}} &\hookrightarrow \Hb^{s_{\tow}} (X), &
                                                        \cY^{s_{\tow}-1}
  &\hookrightarrow \Hb^{s_{\tow}-1}(X), \\
  \cX^{s_{\tow}^{*}} &\hookrightarrow \Hb^{1-s_{\tow}}(X), &
                                                          \cY^{s_{\tow}^{*}}
  &\hookrightarrow \Hb^{-s_{\tow}}(X).
\end{align*}

\section{Propagation of singularities in the bulk}
\label{sec:prop-sing-bulk}

The aim of this section is to prove a regularity result for forward
solutions $u$ of $Lu \in C^{\infty}_{c}(M^{\circ})$.  In particular,
we establish that $u$ lies in a weighted $\Hbdst$ space and enjoys
additional regularity with respect to the $\Psib^{0}(M)$-module
\begin{equation*}
  \module = \{ A \in \Psib^{1}(M) \mid \sigma
  (A)\rvert_{\mathcal{R}^+} = 0\}.
\end{equation*}
The main result of this section is the following proposition:
\begin{proposition}
  \label{prop:output-of-bulk-prop}
  If $u \in \rho^{\ell}\domst$ satisfies $Lu \in
  C^{\infty}_{c}(M^{\circ})$ and $u \equiv 0$ for $t \ll 0$, then
  there are $s, \gamma \in \reals$ so that $s + \gamma < 1/2$ and $u
  \in \Hbdst^{s,\gamma}$.  Moreover, $u$ possesses module regularity
  with respect to this space, i.e., if $A_{1},\dots,A_{N}\in
  \module$, then $A_{1}\dots A_{N} u \in \Hbdst^{s,\gamma}$.
\end{proposition}

Away from the cone points and the future radial set, standard elliptic
regularity and hyperbolic propagation arguments apply to establish
$\Hbdst^{\bullet,\bullet}$ regularity of any order.  Our aim therefore
is to establish the proposition microlocally in these regions.  In
Section~\ref{sec:radial-set-bulk} we recall the propagation estimates
at the radial sets $\mathcal{R}^{\pm}$, while in
Section~\ref{sec:near-singular-points} we establish the necessary
estimates near the singularities.

\subsection{The radial set}
\label{sec:radial-set-bulk}

At $\mathcal{R}^{+}$ (i.e., at $N^{*}S_{+}$), the Hamilton vector
field of $L$ is radial and so we appeal to the radial point
propagation estimates of Vasy~\cite{vasy:microlocal}.\footnote{As we are
  working with the forward solution in the bulk, we have no need for
  the estimates at $\mathcal{R}^{-}$, though these estimates would of
  course be necessary to show that $\Box$ is Fredholm on appropriate
  spaces.}  Though we state the estimates with reference to the domain
$\domst$, this is immaterial as the estimates localize and the radial
sets are disjoint from the conic singularities.

\begin{proposition}[{cf.~\cite[Proposition 5.4]{BVW2}}]
  If $u \in \Hbdst^{-\infty,\ell}(M)$ for some $l$, $Lu \in
  \Hbdst^{m-1,l}$, and $u \in \Hbdst^{m,l}$ on a punctured
  neighborhood $U \setminus \partial \mathcal{R}^{+}$ of
  $\partial\mathcal{R}^{+}$ in $\Sbstar M$, then for $m' \leq m$ with
  $m' + l < 1/2$, we have $w \in \Hbdst^{m',l}(M)$ at $\partial
  \mathcal{R}^{+}$ and for $N \in \naturals$ with $m' + N \leq M$ and
  $A \in \module^{N}$, $Aw$ is in $\Hbdst^{m',l}$ at $\partial
  \mathcal{R}^{+}$.  
\end{proposition}

In particular, if $Lu \in \Hbdst^{\infty,l}$ and $u \in
\Hbdst^{\infty,\ell}$ on a punctured neighborhood of $\partial
\mathcal{R}^{+}$, then as long as $m' + l < 1/2$, $Au \in
\Hbdst^{m',l}$ at $\partial \mathcal{R}^{+}$ for $A \in \module^{N}$.
We remark that as $\pd \mathcal{R}^{+}$ is disjoint from $\cf$,
$\Hbdst^{m',l}$ regularity agrees with $\Hb^{m',l}$ regularity.

\subsection{Near the singular points}
\label{sec:near-singular-points}

For finite times, the work of Melrose--Wunsch~\cite{MW}
establishes the needed propagation results.  We therefore prove the
analogous statement near the intersection $\mf \cap \cf$.  Recall that
the compressed characteristic set is defined in Section~\ref{sec:operator-ps}.
\begin{proposition}
  \label{prop:propagation-bulk}
  If $u\in \rho^{\ell}\domst$ is the forward solution of $Lu =f$ for $f \in
  C^{\infty}_{c}(M^{\circ})$, then $\WFbdst^{m,\ell} u \subset
  \Sigmadot$.  For
  \begin{equation*}
    q_{0} = \{ (\rho = 0, x=0, z \in Z, \taub_{0} =\pm 1, \xib=0,
    \zetab=0)\} \subset \Sigmadot \cap \{ \rho =0\}
  \end{equation*}
  and let $U$ denote a neighborhood of $q_{0}\in \Sigmadot$.  If
  \begin{equation*}
    U \cap \{ \xib / \taub > 0\} \cap \WFbdst^{s,\ell}(u) = \emptyset,
  \end{equation*}
  then
  \begin{equation*}
    q_{0} \cap \WFbdst^{s,\ell}(u) = \emptyset.
  \end{equation*}
\end{proposition}
As the wavefront set is closed, this proposition yields regularity at
the outgoing points ($\xib/\taub < 0$) sufficiently near $q_{0}$.

The first statement (that the wavefront set lies in the characteristic
set) is the main result of Section~\ref{sec:elliptic-regularity-bulk},
while the diffractive theorem (the absence of ``incoming'' wavefront
set implies the absence of ``outgoing'' wavefront set) is proved in
Section~\ref{sec:hyperb-prop-bulk}. 

Throughout the rest of this section we use $Q_{j}$ to denote those
first-order conic differential operators not lying in $\Diffb^{1}$.
We set $Q_{0} = 1/x$, $Q_{1} = D_{x}$, and
$Q_{j} = \frac{1}{x} D_{z_{j}}$ for local coordinates
$z_{2},\dots, z_{n}$ on $Z$.  We further \emph{assume all
  pseudodifferential operators and distributions are localized to a
  region with $x\leq 1/4$.}  As mentioned above, we continue to abuse
notation by using the symbol $\Diffe$ to denote differential operators
that are edge-like at $\cf$ (i.e., in $x$) and otherwise $\bl$-like at
$\mf$ (i.e., in $\rho$).  We measure $L^{2}$ with respect to the
density for the conic-$\bl$-metric $\rho^{2}g$; in local coordinates
this has the form
\begin{equation*}
  \frac{x^{n-1}\sqrt{\met}}{\rho}\,d\rho\,dx\,dz.
\end{equation*}
With respect to this density, we observe that $L$ has the following form:
\begin{equation*}
  L = \left( \rho D_{\rho} + x D_{x}\right)^{*}\left( \rho D_{\rho} +
    x D_{x}\right) - D_{x}^{*}D_{x} - \left(
    \frac{1}{x}\grad_{z}\right)^{*}\left( \frac{1}{x}\grad_{z}\right)
  - \frac{n^{2}-1}{4}.
\end{equation*}

\subsubsection{Elliptic regularity}
\label{sec:elliptic-regularity-bulk}

The elliptic part of Proposition~\ref{prop:output-of-bulk-prop}
follows from a main lemma and the ellipticity of the operator away
from $\Sigmadot$.  Before stating the main lemma, we introduce for
brevity the shorthand notation
\begin{equation*}
  \abs{d_{x,z}f}^{2} = \abs{\pd[x]f}^{2} + \abs{x^{-1}\grad_{z}f}^{2},
\end{equation*}
where the latter norm is measured with respect to the metric $\met$ on
$Z$.  

The main estimate follows by pairing $Lv$ with $v$ for a family of $v$
and then integrating by parts; its proof is essentially identical to
the one given by Melrose--Vasy--Wunsch~\cite[Lemma 8.9]{MVW} with
a minor modification we will describe below.
\begin{lemma}[{cf.~\cite[Lemma 8.9]{MVW}}]
  \label{lemma:other-ell-bulk}
  Suppose that $K\subset U \subset \Sbstar M$ with $K$ compact and $U$
  open, and suppose further that $A_{r}$ constitute a bounded family
  of basic elements of $\Psibinf$ with $\WFb'(A_{r})\subset K$ in the
  sense of uniform wavefront sets of families, and $A_{r} \in
  \Psib^{s-1}$ for all $r\in (0,1)$.  There exist $G \in
  \Psib^{s-1/2}$ and $\tG\in \Psib^{s}$ with $\WFb'(G),
  \WFb'(\tG)\subset U$ and $C_{0} > 0$ so that for all $\epsilon > 0$,
  $r \in (0,1)$, and $u \in \domst$ with $\WFbdst'^{s-1/2}(u) \cap U =
  \emptyset$ and $\WFbdprime^{s}(Lu)\cap U = \emptyset$, we have
  \begin{align*}
    &\left|\int\left(\abs{d_{x,z}A_{r}u}^{2} +
    \frac{n^{2}-1}{4}\abs{A_{r}u}^{2} - \abs{(\rho \pd[\rho] + x
    \pd[x])A_{r}u}^{2}\right)\frac{x^{n-1}\sqrt{\met}}{\rho}\,d\rho\,dx\,dz\right|
    \\
    &\quad \leq \epsilon \left( \norm{d_{x,z}A_{r}u}_{L^{2}}^{2} +
      \norm{\rho\pd[\rho]A_{r}}_{L^{2}}^{2}\right) + C_{0} \left(
      \norm{u}_{\domst}^{2} + \norm{Gu}_{\domst}^{2} + \epsilon^{-1}
      \norm{Lu}_{\domst'}^{2} + \epsilon^{-1}\norm{\tG Lu}_{\domst'}^{2}\right).
  \end{align*}
\end{lemma}

After observing that for $v\in \domst$, 
\begin{equation*}
  \ang{Lv,v} = \norm{(\rho \pd[\rho] + x\pd[x])v}^{2} - \norm{\pd[x]
    v}^{2} - \norm{\frac{1}{x}\grad_{z}v}^{2} - \frac{n^{2}-1}{4}\norm{v}^{2},
\end{equation*}
the proof of Lemma~\ref{lemma:other-ell-bulk} is identical to its
counterpart in the work of Melrose--Vasy--Wunsch in \cite{MVW} with
$\pd[t]$ replaced by $\rho \pd[\rho] + x\pd[x]$.

At this stage, we record a corollary useful in the next subsection:
\begin{corollary}
  \label{cor:elliptic-control}
  Under the hypotheses of Lemma~\ref{lemma:other-ell-bulk}, we can
  estimate the domain norm of $A_{r}u$ by
  \begin{equation*}
    \norm{A_{r}u}_{\domst} \leq C \left( \norm{u}_{\domst} +
      \norm{Gu}_{\domst} + \norm{Lu}_{\domst'} + \norm{\tG
        Lu}_{\domst'} + \norm{(\rho \pd[\rho] + x\pd[x])A_{r}u}_{L^{2}}\right).
  \end{equation*}
\end{corollary}
Corollary~\ref{cor:elliptic-control} allows us to replace factors of
$Q_{j}$ with the $\bl$-differential operator $\rho \pd[\rho] +
x\pd[x]$ at the cost of terms already on the right side of
Lemma~\ref{lemma:other-ell-bulk}.  In other words, we can control the
$\domst$ norm of $Au$ by the $\Hb^{1}$ norm of $Au$ and the other
terms on the right.

We conclude this section with the proof of the first part of
Proposition~\ref{prop:propagation-bulk}, namely that
$\WFbdst^{m,\ell}u \subset \Sigmadot$.  We employ a simpler version of
the argument used by Melrose--Vasy--Wunsch~\cite[Proposition
8.10]{MVW}.

Suppose $q \in \Sbstar M \setminus \Sigmadot$.  For finite times
(i.e., $\rho > 0$, the theorem of Melrose--Wunsch in \cite{MW} applies
and so we may assume $q$ projects to $\rho = 0$.  Likewise, standard
elliptic arguments apply away from $x=0$ and so we may assume $q$
projects to $x=0$, so that
\begin{equation*}
  q = (\rho = 0, x = 0, z \in Z, \taub, \xib, \zetab),
\end{equation*}
where $\xib^{2} + \abs{\zetab}^{2} > 0$.  We assume inductively that
$q\notin \WFbdst^{s-1/2,\ell}(u)$ and aim to show that $q\notin
\WFbdst^{s,\ell}(u)$.  Let $A\in \Psib^{s,\ell}$ be a basic operator
so that
\begin{enumerate}
\item $\ds \WFb'(A) \cap \WFbdst^{s-1/2,\ell}(u) =\emptyset$, and 
\item $\ds \WFb'(A)$ is a subset of a small neighborhood $U$ of $q$ on
  which $\hat{\xib}^{2} + \abs{\hat{\zetab}}^{2} > c > 0$.
\end{enumerate}

We now introduce $\Lambda_{r} \in \Psib^{-2}$ for $r > 0$ with symbol
$(1+r(\taub^{2} + \xib^{2} + \abs{\zetab}^{2}))^{-1}$ so that
$\Lambda_{r} \in \Psibinf^{0}$ uniformly and $\Lambda_{r} \to \Id$ as
$r\to 0$.  We set $A_{r} = \Lambda_{r}A$ so that for $r > 0$, we have
\begin{equation*}
  \sigma (A_{r}) = \frac{a}{1 + r(\taub^{2} + \xib^{2} + \abs{\zetab}^{2})},
\end{equation*}
where $a$ is the symbol of $A$ and $A_{r}\rho^{\ell}$ and
$\rho^{-\ell}u$ satisfy the hypotheses of
Lemma~\ref{lemma:other-ell-bulk}.

By the Lemma~\ref{lemma:other-ell-bulk}, the difference
\begin{equation*}
  \norm{\pd[x]A_{r}u}^{2} + \norm{\frac{1}{x}\grad_{z}A_{r}u}^{2} +
  \frac{n^{2}-1}{4}\norm{A_{r}u}^{2} - \norm{(\rho \pd[\rho] +
    x\pd[x])A_{r}u}^{2} - \epsilon \norm{d_{(x,z)}A_{r}u}^{2}
\end{equation*}
is uniformly bounded in $r$.  Writing this quantity as
\begin{equation*}
  \frac{1-\epsilon}{2} \left( \norm{\pd[x]A_{r}u}^{2} +
    \norm{\frac{1}{x}\grad_{z}A_{r}u}\right) + I,
\end{equation*}
we now show $I \geq 0$.

Indeed, we observe that if $\delta > 0$ is sufficiently small, then
the operator $B \in \Psib^{1}$ with principal symbol given by 
\begin{equation*}
  \sigmab(B) = \left( \frac{1-\epsilon}{\delta^{2}}\left( \xib^{2} +
      \zetab^{2}\right) - (1+\epsilon) \left( \taub + \xib\right)^{2}\right)^{1/2}
\end{equation*}
is elliptic on $U$.

Moreover, if $A$ is supported in $\{ x < \delta\}$, then
\begin{align*}
  \norm{\pd[x]A_{r}u}^{2} &\geq
                            \frac{1}{\delta^{2}}\norm{x\pd[x]A_{r}u}^{2}
                            , \\
  \norm{\frac{1}{x}\grad_{z}A_{r}u}^{2} & \geq \frac{1}{\delta^{2}}\norm{\grad_{z}A_{r}u}^{2},
\end{align*}
and so by shrinking the support of $A$, $I$ is bounded below by
\begin{align*}
  I &\geq \frac{1-\epsilon}{\delta^{2}}\left( \norm{x\pd[x]A_{r}u}^{2}
      + \norm{\frac{1}{x}\grad_{z}A_{r}u}^{2}\right) - (1 + \epsilon)
      \norm{(\rho \pd[\rho] + x \pd[x])A_{r}u}^{2} \\
 &  = \norm{BA_{r}u}^{2} + \ang{FA_{r}u, A_{r}u},
\end{align*}
where $B, F\in \Psib^{1}$ and $B$ has principal symbol given above.
As $F$ is order $1$ and $Au\in \domst$, the second term is
uniformly bounded in $r$.

As $I$ is bounded below, we deduce that
\begin{equation*}
  \frac{1-\epsilon}{2} \left( \norm{\pd[x]A_{r}u}^{2}+ \norm{\frac{1}{x}\grad_{z}A_{r}u}^{2}\right)
\end{equation*}
is uniformly bounded in $r$.  Extracting weak limits shows that $Au
\in \domst$ and proves the first part of the proposition.

\subsubsection{Hyperbolic propagation}
\label{sec:hyperb-prop-bulk}

The aim of this subsection is to complete the proof of
Proposition~\ref{prop:propagation-bulk}.  We proceed by a positive
commutator estimate; the positivity essentially stems from the
commutator of $L$ with $x\pd[x]$.  We treat the case of $\taub_{0} >
0$ here; the other case follows by flipping the sign of $\xib$.  Indeed, for
\begin{equation*}
\hat{\xib} = \frac{1}{\taub}\xib, \quad p_{0} = \sigmab(L),
\end{equation*}
the Hamilton vector field of $p_{0}$ satisfies
\begin{equation*}
  \frac{1}{2} H_{p_{0}} ( - \hat{\xib}) =
  \frac{1}{x^{2}}\left( \xib^{2} + \abs{\zetab}^{2}\right).
\end{equation*}

As in Vasy~\cite{Vasy08}, we define two auxiliary functions
\begin{equation*}
  \omega = x^{2}+ \rho^{2},
\end{equation*}
and
\begin{equation*}
  \phi = - \hat{\xib} + \frac{1}{\beta^{2}}\delta \omega,
\end{equation*}
where $\beta > 0$ is a parameter to be chosen.  The first function
acts as a localizer near the corner, while the second function
provides the positivity in the estimate.  As long as
$\omega < \delta$, we can bound
\begin{equation*}
  \frac{1}{\tau}H_{p_{0}}\omega = \bo \left( \sqrt{\omega} \left(
      \frac{\hat{\xib}^{2}}{x^{2}} +
      \frac{\abs{\hat{\zetab}}^{2}}{x^{2}} + 1 \right)^{1/2}\right).
\end{equation*}

We now fix three smooth functions of one variable $\chi_{0}$,
$\chi_{1}$, and $\chi_{2}$.  We demand that $\chi(s) = \exp (-1/s)$
for $s>0$ so that $\chi_{0}'(s) = s^{-2}\chi_{0}(s)$.  We take
$\chi_{1}$ supported in $[0,\infty)$ to be equal to $1$ on
$[1,\infty)$ and so that $\chi_{1}' \geq 0$ is compactly supported in
$(0,1)$.  Finally, for a given parameter $c_{1}$, we take $\chi_{2} \in
C^{\infty}_{c}(\reals)$ supported in $[-2c_{1},2c_{1}]$ and
identically $1$ on $[-c_{1},c_{1}]$.  We insist that all cut-off
functions and their derivatives have smooth square roots up to sign.

With $\chi_{\bullet}$ in hand, we finally define the basic test symbol
$a$ by
\begin{equation*}
  a = \chi_{0} \left( 1 - \frac{\phi}{\delta}\right) \chi_{1} \left(
    \frac{-\hat{\xib}}{\delta} + 1 \right) \chi_{2}\left(
    \hat{\xib}^{2} + \abs{\hat{\zetab}}^{2}\right),
\end{equation*}
where $\delta > 0$ is another parameter to be chosen.

As in Melrose--Vasy--Wunsch~\cite{MVW} and
Gannot--Wunsch~\cite{GW}, we can arrange that $a$ is well-localized
near $q_{0}= \{ (\rho = 0, x=0, z \in Z, \taub=\pm 1, \xib=0, \zetab=0
)\}$.
\begin{lemma}
  Given any neighborhood $U$ of $q_{0}$ and any $\beta > 0$, there are
  $\delta_{0}> 0$ and $c_{1}> 0$ so that $a$ is supported in $U$ for
  all $0 < \delta < \delta_{0}$.
\end{lemma}

We now choose a basic operator $B \in \Psib^{1/2}$ with
\begin{equation*}
  b = \sigmab(B) = \tau^{1/2}\delta^{-1/2} (\chi_{0}\chi_{0}')^{1/2}\chi_{1}\chi_{2},
\end{equation*}
so that, when taking derivatives of $a$, those falling on $\chi_{0}$
yield factors of $b^{2}$.  We further choose $C \in \Psib^{0}$ with
principal symbol
\begin{equation*}
  \sigmab(C) = \frac{\sqrt{2}}{\tau}\abs{\tau + \xib}\psi,
\end{equation*}
where $\psi \in S^{0}(\Tbstar M)$ is identically $1$ on the support of
the symbol of $B$.

We can now compute the commutator of $A^{*}A$ and $L$:
\begin{lemma}[{cf.~\cite[Lemma 9.6 and Theorem 9.7]{MVW}}]
  \label{lemma:commutator-bulk-calc}
  There is a $\delta_{0} > 0$ so that for all $0 < \delta <
  \delta_{0}$, the commutator of $L$ and $A^{*}A$ is given by
  \begin{equation*}
    i \left[ A^{*}A, L\right] = R' L + B^{*}\left( C^{*}C + R_{0} +
      \sum_{j} R_{j}Q_{j} + \sum_{j,k}Q_{j}^{*}R_{jk}Q_{k}\right)B +
    R'' + E' + E'',
  \end{equation*}
  where the terms enjoy the following properties:
  \begin{itemize}
  \item all factors are microlocalized near $q_{0}$, 
  \item $R_{0} \in \Psib^{0}$, $R', R_{j}\in \Psib^{-1}$, $R_{jk}\in
    \Psib^{-2}$,
  \item $E'$, $E'' \in x^{-2}\Diffe^{2} \Psib^{-1}$, $R'' \in
    x^{-2}\Diffe^{2}\Psib^{-2}$, 
  \item the symbols $r_{0}$, $r_{j}$< and $r_{jk}$ of $R_{0}$,
    $R_{j}$, and $R_{jk}$ are supported in $\{ \omega \leq 9
    \delta^{2}\beta\}$,
  \item the symbols $r_{0}$, $\taub r_{j}$, and $\taub^{2}r_{jk}$ are
    bounded by both
    \begin{equation*}
      c \left( 1 + \frac{1}{\beta^{2}\delta}\right), \quad
      \text{and}\quad c \left( \delta \beta + \beta^{-1}\right),
    \end{equation*}
  \item $\ds\WFb'(E') \subset \xib^{-1}((0,\infty))\cap U$, and
  \item $\ds \WFb'(E'')\cap \Sigmadot = \emptyset$.
  \end{itemize}
\end{lemma}

\begin{proof}
  The principal symbol of the commutator is given by the action of the
  Hamilton vector field of $p_{0}$ on $a^{2}$; the choice of the
  function $\chi_{0}$ ensures that when derivatives fall on this term,
  we obtain the contributions sandwiched between $B^{*}$ and $B$.  The
  positive term arises from the near homogeneity (in $x$) of $L$.
  Indeed, we exchange the leading term in $a\pd[\xi]a$ with $L$,
  leaving the symbol $\abs{\taub + \xib}^{2}/\taub^{2}$ and obtaining
  the $C^{*}C$ term as well as the $R'L$ term.

  Derivatives falling on $\chi_{1}$ give contributions to the $E'$
  term; those falling on $\chi_{2}$ provide contributions to $E''$.
  Commuting the $Q_{j}$ through $B$ also leads to contributions to
  $E'$ and $E''$.

  The $R''$ term arises as the computation occurs only at the
  principal symbol level; this term is also used to further absorb
  other lower order commutation terms.
\end{proof}

We also observe that we can estimate the remainder terms via the
symbol calculus:
\begin{lemma}
  \label{lemma:bulk-remainder-terms}
  Given $\epsilon > 0$, there is a $\delta_{1}\in (0,\delta_{0})$ so
  that for all $0 < \delta < \delta _{1}$, and all $v \in \domst$,
  \begin{align*}
    &\abs{\ang{R_{0}Bv,Bv}} + \sum_{j}\abs{\ang{R_{j}Q_{j}Bv,Bv}} +
      \sum_{j,k}\abs{\ang{Q_{j}^{*}R_{jk}Q_{k}Bv,Bv}} \\
    &\quad \leq \epsilon\norm{Bv}^{2} + C\norm{R'Bv}^{2} + C\left(
      \norm{u}_{\domst}^{2} + \norm{Gu}_{\domst}^{2} +
      \norm{Lu}_{\domst'}^{2} + \norm{\tG L u}_{\domst '}^{2}\right)
  \end{align*}
  for some $R' \in \Psib^{-1}$.
\end{lemma}

\begin{proof}
  The lemma follows from the symbol estimates of
  Lemma~\ref{lemma:commutator-bulk-calc} and
  Corollary~\ref{cor:elliptic-control}, together with the observation
  that for $A\in \Psib^{0}$, there is an $A' \in \Psib^{-1}$ so that
  for all $u \in L^{2}$
  \begin{equation*}
    \norm{Au} \leq \sup \abs{\sigmab(A)}\norm{v} + C\norm{A'v}.
  \end{equation*}
\end{proof}

We now finish the proof of Proposition~\ref{prop:propagation-bulk}.
\begin{proof}[Proof of Proposition~\ref{prop:propagation-bulk}]
  We first consider the case of $\ell = 0$.  Suppose $s < \sup \{ s' :
  q_{0} \notin \WFbdst^{s'}u\}$; shrinking $U$ if necessary we may
  assume $U \cap \WFbdst^{s}(u) =\emptyset$.  Our aim is to show $q
  \notin \WFbdst^{s+1/2}(u)$.

  As we measure regularity with respect to $\domst$, we know that if
  $B\in \Psib^{s}$ localizes to $U$, then $Bu$, $Q_{i}Bu$, and
  $\rho\pd[\rho]Bu$ all lie in $L^{2}$.  By the hypothesis and
  Corollary~\ref{cor:elliptic-control}, it suffices to control $\rho
  \pd[\rho] Bu$ at $q_{0}$.  In particular, it suffices to find a
  $\bl$-pseudodifferential operator of order $s+3/2$ that is elliptic
  at $q_{0}$ and for which $Bu\in L^{2}$.  (This explains the apparent
  shift in order by one below.)

  Let $A$, $B$, and $C$ be as in the discussion preceding
  Lemma~\ref{lemma:commutator-bulk-calc} and let $\Lambda_{r}$ be a
  quantization of
  \begin{equation*}
    \taub^{s+1}\left( 1 + r \taub^{2}\right)^{-(s+1)/2}, \quad r \in [0,1],
  \end{equation*}
  and set $A_{r} = A\Lambda_{r}\in \Psib^{0}$ for $r > 0$ and $A_{r}$
  is uniformly bounded in $\Psibinf^{s+1}$.  We may further arrange
  that $[L, \Lambda_{r}] =0$.

  By the calculation in Lemma~\ref{lemma:commutator-bulk-calc}, we may
  write
  \begin{align}
    \label{eq:bulk-full-terms}
    i\ang{[A_{r}^{*}A^{*}, L]u, u} &= \norm{CB\Lambda_{r}u}^{2} +
                                     \ang{R' L\Lambda_{r}u,
                                     \Lambda_{r}u} +
                                     \ang{R_{0}B\Lambda_{r}u,
                                     \Lambda_{r}u} \\
                                   &\quad +
                                     \sum_{j}\ang{R_{j}Q_{j}B\Lambda_{r}u,
                                     B\Lambda_{r}u} +
                                     \sum_{j,k}\ang{R_{jk}Q_{j}B\Lambda_{r}u,
                                     Q_{k}B\Lambda_{r}u} \notag\\
                                   &\quad + \ang{R'' \Lambda_{r}u,
                                     \Lambda_{r}u} +
                                     \ang{(E'+E'')\Lambda_{r}u, \Lambda_{r}u}.\notag
  \end{align}
  As $u \in \domst$, the pairing on the left is well-defined:
  \begin{align*}
    \ang{[A_{r}^{*}A, L]u, u} &= \ang{A_{r}Lu, A_{r}u} - \ang{A_{r}u, A_{r}Lu}.
  \end{align*}
  As $Lu$ is residual, these terms are uniformly bounded in $r$ and so
  we may estimate $\norm{CB\Lambda_{r}u}^{2}$ by the other terms in
  equation~\eqref{eq:bulk-full-terms}.  The second term is uniformly
  bounded because $Lu$ is residual, while the next three terms are
  estimated by Lemma~\ref{lemma:bulk-remainder-terms}.  The $R''$ term
  is bounded by the regularity hypothesis of $u$ on $U$, while the
  $E''$ term is bounded by elliptic regularity.  Finally, the $E'$ term
  is bounded by the hypothesis of the theorem.  We can therefore find
  a constant $\mathcal{C}$ independent of $r$ so that
  \begin{equation*}
    \norm{CB\Lambda_{r}u}^{2}\leq \mathcal{C} +
    \epsilon\norm{B\Lambda_{r}u}^{2} + \mathcal{C}\left(
      \norm{R'Bv}^{2} + \norm{u}_{\domst}^{2} + \norm{Gu}_{\domst}^{2}
      + \norm{Lu}_{\domst'}^{2} + \norm{\tG Lu}_{\domst'}^{2}\right),
  \end{equation*}
  where $G\in \Psib^{s+1/2}$, $\tG \in \Psib^{s+1}$ are supported in
  $U$.  An application of the symbol calculus shows that
  $\norm{CB\Lambda_{r}u}$ (and the rest of the right side) controls
  $\norm{B\Lambda_{r}u}$.  The other terms on the right are uniformly
  bounded by the assumed regularity of $u$, so we can extract a
  subsequence and conclude that $B\Lambda_{0}u \in L^{2}$, so that
  $q_{0}\notin \WFbdst^{s+1/2}(u)$.  By iteratively shrinking the
  neighborhoods $U$, one can then show that in fact
  $q_{0}\notin \WFbdst^{\infty}(u)$.

  Finally, we now suppose that $\ell \neq 0$.  As $Lu \in
  C^{\infty}_{c}(M^{\circ})$, we can apply the above argument to $v =
  \rho^{\ell}u$ and $\tilde{L} = \rho^{\ell}L \rho^{-\ell}$.  As $L$
  and $\tilde{L}$ differ only by an element of $\Diffb^{1}$, the same
  proof applies to $v$.
\end{proof}

\section{The boundary operator $\Ps$}
\label{sec:prop-sing}

The aim of this section is to establish the mapping properties of
$\Ps$ (recall that $\Lambda^{\pm}$ are the radial sets for $\Ps$ and
are the fiber infinities of $N^{*}S_{\pm}$):
\begin{proposition}
  \label{prop:mapping-props}
  The family $\Ps$ has the following mapping properties:
  \begin{enumerate}
  \item $\ds \Ps : \cX^{s_{\tow}}\to \cY^{s_{\tow}-1}$ and $\ds\Ps^{*} :
    \cX^{s_{\away}^{*}} \to \cY^{s_{\away}^{*}-1}$ are Fredholm.
  \item The operators $\Ps$ form a holomorphic Fredholm family on
    these spaces in
    \begin{equation*}
      \complexes_{s_{+}, s_{-}} = \left\{ \sigma \in \complexes \mid
        s_{+} < \frac{1}{2} + \Im \sigma < s_{-}\right\},
    \end{equation*}
    with $s_{\tow}|_{\Lambda^{\pm} }= s_{\pm}$.  The formal adjoint
    $\Ps^{*}$ is antiholomorphic in the same region.
  \item The inverse $\Ps^{-1}$ has only finitely many poles in each
    strip $a< \Im \sigma < b$.
  \item For all $a$ and $b$, there is a constant $C$ so that
    \begin{equation}
      \label{eq:semiclass-est}
      \norm{\Ps^{-1}}_{\cY^{s_{\tow}-1}_{\abs{\sigma}^{-1}} \to
        \cX^{s_{\tow}}_{\abs{\sigma}}}\leq C \langle \Re \sigma \rangle^{-1}
    \end{equation}
    on $a < \Im \sigma < b$, $\abs{\Re \sigma} > C$, with a similar
    estimate holding for $(\Ps^{*})^{-1}$.
  \item The set of poles of $\Ps^{-1}$ (and $(\Ps^{*})^{-1}$) is
    independent of the choice of $s_{\tow}$. 
  \end{enumerate}
\end{proposition}

The first two parts of Proposition~\ref{prop:mapping-props} follow
from a sequence of propagation estimates; the second two parts follow
from semiclassical analogues of the same sorts of estimates.  For the
Fredholm statement, we propagate regularity \emph{out of} $S_{-}$ via
radial point estimates (as the $\cX$ spaces are more regular than the
threshold there), then rely on standard hyperbolic propagation
estimates to carry this regularity to a neighborhood of $S_{+}$, where
we then finish the propagation argument with below-threshold radial
point estimates.  The regularity in $C_{\pm}$ is treated by means of
the elliptic theory on cones, as $\Ps$ is classically elliptic there.
In the semiclassical case, however, the semiclassical characteristic
set of $\Ps$ extends into $C_{\pm}$ and we establish a semiclassical
diffractive estimate to carry the regularity of the solution through
the singularity of the operator.

Many of the microlocal estimates employed to establish
Proposition~\ref{prop:mapping-props} are already in the literature;
the main missing components are the Fredholm statement
(Section~\ref{sec:fredholm-property}) and the semiclassical
propagation estimate for $\Ps$ near the cone point
(Sections~\ref{sec:ellipt-regul-near} and~\ref{sec:hyperb-prop}).  The
last part of Proposition~\ref{prop:mapping-props} follows from
standard arguments in the resonances literature.

\subsection{The Fredholm property}
\label{sec:fredholm-property}

We first show that $\Ps$ is Fredholm on the desired spaces (parts 1
and 2 of Proposition~\ref{prop:mapping-props}).  In particular, we
prove the following proposition:
\begin{proposition}
  \label{prop:Fredholm}
  Given $s_{\pm}$ and $s_{\tow/\away} |_{\Lambda^{\pm}} = s_{\pm}$,
  there is

  \begin{equation*}
    \norm{u}_{\cX^{s_{\tow}}} \leq C\left( \norm{\Ps
        u}_{\cY^{s_{\tow}}-1} + \norm{u}_{\Hbdst^{-N}}\right).
  \end{equation*}
  
\end{proposition}

Away from the radial sets and the conic singularity $\mf\cap \cf$,
standard elliptic regularity and hyperbolic propagation arguments can
be pieced together.  Near the singularity at the poles, we appeal to
the following elementary lemma, which follows essentially immediately
after integrating by parts:
\begin{lemma}
  Fix $\chi \in \CI(\mf)$ supported in $\{ x < 1/4\}$.  For any $N$,
  there is a constant $C$ so that
  \begin{equation*}
    \norm{\chi u}_{\dom} \leq C \left( \norm{\Ps (\chi u)}_{L^{2}} +
      \norm{\chi u}_{L^{2}}\right).
  \end{equation*}
  In particular, for all $s$ and all $N$, we may estimate
  \begin{equation*}
    \norm{\chi u}_{\cX^{s_{\tow}}} \leq C \left( \norm{\Ps (\chi
        u)}_{\cY^{s_{\tow}-1}} + \norm{\chi u}_{\cY^{s_{\tow}-1}}\right).
  \end{equation*}
\end{lemma}

The same argument provides a similar estimate for $\Ps^{*}$ in the
appropriate (dual) spaces.

Near the radial sets $\Lambda^{\pm}$, the radial point estimates of
Vasy~\cite{vasy:microlocal} (building on work of Melrose~\cite{Melrose94}) apply without change:
\begin{lemma}[{\cite[Propositions~2.3 and~2.4]{vasy:microlocal}}]
  \label{lem:radial-points-bdry}
  For all $N$ and for $s_{0} > m > \frac{1}{2} + \Im \sigma$, and for
  all $A, B, G\in \Psib^{0}(\mf)$ supported near $\Lambda^{-}$ with
  $A, G$ elliptic at $\Lambda^{-}$ and so that all bicharacteristics
  from the microsupport of $B$ tend to $\Lambda^{-}$ in one direction
  while remaining in the elliptic set of $G$, we have
  \begin{equation*}
    \text{If }Au \in H^{m} \text{ then }\norm{Bu}_{H^{s_{0}}} \leq
    C\left( \norm{G \Ps u}_{H^{s_{0}-1}} + \norm{u}_{\Hbd^{-N}}\right).
  \end{equation*}
  
  For $s_{0} < \frac{1}{2} + \Im \sigma$ and $A,B,G\in \Psib^{0}(\mf)$
  supported near $\Lambda^{+}$ with $B,G$ elliptic at $\Lambda^{+}$ so
  that all bicharacteristics from $\WF'(B) \setminus \Lambda^{+}$
  reach the microsupport of $A$ in one direction while remaining in
  the elliptic set of $G$, we have
  \begin{equation*}
    \norm{Bu}_{H^{s_{0}}} \leq C \left( \norm{G\Ps u}_{H^{s_{0}-1}} +
      \norm{Au}_{H^{s_{0}}} + \norm{u}_{H^{-N}}\right).
  \end{equation*}
\end{lemma}

An analogous theorem holds for $\Ps^{*}$ with $\sigma$ replaced by its
complex conjugate and the direction of propagation reversed (so that
the roles of $\Lambda^{\pm}$ are exchanged).

Taking microlocal partitions of unity as appropriate, we therefore
have the two estimates
\begin{align*}
  \norm{u}_{\cX ^{s_{\tow}}} &\leq C \left( \norm{\Ps
                               u}_{\cY^{s_{\tow}-1}} + \norm{u}_{\cY^{s_{\tow}-1}}\right),\\
  \norm{u}_{\cX^{s_{\tow}^{*}}} &\leq C\left(
                               \norm{\Ps^{*}u}_{\cY^{s_{\tow}^{*}-1}} + \norm{u}_{\cY^{s_{\tow}^{*}-1}}\right).
\end{align*}
As the inclusions $\cX^{s_{\tow}}\hookrightarrow \cY^{s_{\tow}-1}$ and
$\cX^{s_{\tow}^{*}}\hookrightarrow \cY^{s_{\tow}^{*}-1}$ are compact,
the operators $\Ps$ and $\Ps^{*}$ are Fredholm in the stated spaces,
proving the first part of the theorem.  The second part of the theorem
follows from the facts that the coefficients of $\Ps$ are holomorphic
and that $s_{\tow}$ may be chosen to satisfy the desired properties
for all $\sigma$ in such a strip.

\subsection{Semiclassical estimates}
\label{sec:propagation-at-conic}

The third part of Proposition~\ref{prop:Fredholm} follows from the
fourth; this rest of this section is devoted to proving the estimate
there.

As the estimate~\eqref{eq:semiclass-est} is a semiclassical hyperbolic
estimate, we work semiclassically with $h = \abs{\sigma}^{-1}$ as our
semiclassical parameter and $\Ph = h^{2} \Ps$.
In these terms, the estimate~\eqref{eq:semiclass-est} is immediately
implied by an estimate of the form
\begin{equation}
  \label{eq:semiclass-est-2}
  \norm{u}_{\cX_{h}^{s_{\tow}}} \leq \frac{C}{h}\norm{\Ph
    u}_{\cY_{h}^{s_{\tow}-1}} + Ch \norm{u}_{\cX_{h}^{s_{\tow}}}
\end{equation}
for some $N$, together with an analogous estimate for $\Ph^{*}$ on the
appropriate spaces.  Away from $\Lambda^{\pm}$ and from $\{ x=0\}$,
the microlocal version of the estimate follows from standard elliptic
regularity and hyperbolic propagation estimates.

\subsubsection{The radial set}
\label{sec:radial-set}

Near $\Lambda^{\pm}$, the estimate follows from a semiclassical
version of radial propagation estimates as in earlier
work~\cite{BVW1, BVW2, vasy:microlocal}.

\begin{proposition}[{cf.~\cite[Propositions 2.8 and 2.9]{vasy:microlocal}}]
  \label{prop:semicl-radial-est}
  For $s|_{\Lambda^{-}} > m > \frac{1}{2} + \Im \sigma$ and $A,B,G\in
  \Psibh^{0}$ supported near $\Lambda^{-}$ with $A,G$ elliptic at
  $\Lambda^{-}$ and so that semiclassical bicharacteristics from the
  microsupport of $B$ tend to $\Lambda^{-}$ in one direction with
  closure in the elliptic set of $G$, we have
  \begin{equation*}
    \text{If }Au \in H^{m}, \text{ then }
    \norm{Bu}_{\cX_{h}^{s^{\tow}}}\leq \frac{C}{h}\norm{G\Ph
      u}_{\cY_{h}^{s_{\tow}-1}} + Ch \norm{u}_{\cY^{s_{\tow}-1}}.
  \end{equation*}

  For $s|_{\Lambda^{+}} < \frac{1}{2} + \Im \sigma$, and for all
  $A,B,G\in \Psibh^{0}$ supported near $\Lambda^{+}$ with $B,G$
  elliptic at $\Lambda^{+}$ and so that semiclassical
  bicharacteristics from $\WFbh'(B) \setminus \Lambda^{+}$ reach the
  microsupport of $A$ in one direction while remaining in the elliptic
  set of $G$, we have
  \begin{equation*}
    \norm{Bu}_{\cX_{h}^{s_{\tow}}} \leq \frac{C}{h}\norm{G\Ph
      u}_{\cY_{h}^{s_{\tow}-1}} + C \norm{Au}_{\cX_{h}^{s_{\tow}}} +
    Ch \norm{u}_{\cY^{s_{\tow}-1}}.
  \end{equation*}
\end{proposition}
Analogous estimates hold for $\Ph^{*}$ on the dual spaces as well
(with the roles of $\Lambda^{\pm}$ interchanged).

\subsubsection{Elliptic regularity near the singularity}
\label{sec:ellipt-regul-near}

We now consider the problem in the region $\{x
<c_{0}\}$ where $s_{\tow} =1$; we
assume all pseudodifferential operators and distributions are
supported in this region.  Here we have an explicit expression for $\Ps$:
\begin{equation*}
  \Ps = \left( \sigma + xD_{x}\right)^{2} - ni \left( \sigma +
    xD_{x}\right) - D_{x}^{2} + \frac{(n-1)i}{x}D_{x} -
  \frac{1}{x^{2}}\lap_{\met} - \frac{n^{2}-1}{4}.
\end{equation*}
After rescaling and letting $\lambda =
\sigma / \abs{\sigma}$, we have
\begin{equation*}
  \Ph = h^{2}\Ps = \left( \lambda + hxD_{x}\right)^{2} - nih\left(
    \lambda + hxD_{x}\right) - h^{2}D_{x}^{2} +
  \frac{(n-1)ih}{x}hD_{x} - \frac{h^{2}}{x^{2}}\lap_{\met} - \frac{n^{2}-1}{4}h^{2}.
\end{equation*}
As we are only ever concerned with $\Im \sigma \in [a,b]$ for some
fixed $a,b$, we observe that $\lambda = \pm 1 + \bo (h)$.

We prove the estimate near $x=0$ in two main steps; we first consider
the microlocally elliptic region (i.e., away from the characteristic
set) and then the hyperbolic region (near the characteristic set).  In
this section and the next, we consider only the forward problem (for
$\Ph$); the adjoint problem (for $\Ph^{*}$) proceeds nearly
identically, though with a shift downward in the norms considered
(i.e., $\domh$ replaced by $L^{2}$ and $L^{2}$ replaced by $\domh'$).

The main elliptic estimate near the singularity is the following proposition:
\begin{proposition}
  \label{prop:main-ell-est-semicl}
  Suppose $A \in \Psibh^{0}$ is basic and satisfies $\WFbh'(A) \cap
  \Sigmah = \emptyset$.  For any $G\in \Psibh^{0}$ with
  $\WFbh'(A)\subset \ellbh(G)$, there is a constant $C$ so that
  \begin{equation*}
    \norm{Au}_{\cX_{h}^{s_{\tow}}} \leq C\norm{G \Ph u}_{\cY_{h}^{s_{\tow}-1}}
    + C h^{1/2}\norm{Gu}_{\cX^{s_{\tow}}_{h}} + \bo (h^{\infty})\norm{u}_{\cX^{s_{\tow}}_{h}}.
  \end{equation*}
\end{proposition}
By enlarging the microsupport of $G$, one can improve the factor of
$h^{1/2}$ to $h^{N}$ for any fixed $N$.

Integration by parts allows us to prove the following lemma, which
reduces the problem of controlling the domain norm of solutions to
controlling a $\bl$-norm.
\begin{lemma}
  \label{lemma:integration-by-parts-h}
  Suppose $A,G\in \Psibh^{0}$ with $A$ a basic operator satisfying
  $\WFbh'(A) \subset \ellb(G)$.  There is a constant $C$ so that
  \begin{align*}
    &\int\left( \abs{hD_{x}Au}^{2} + \abs{\frac{1}{x}\grad_{z}Au}^{2} -
    \abs{(h xD_{x} + \lambda)Au}^{2} \right)
      x^{n-1}\,dx\,\operatorname{dvol}_{\met} \\
    &\quad\quad \leq \epsilon \norm{Au}_{\domh} ^{2} +
      \frac{C}{\epsilon}\norm{G\Ps u}_{\domh'}^{2} + Ch
      \norm{Gu}_{\domh}^{2} + \bo (h^{\infty})\norm{u}_{\domh}^{2}
  \end{align*}
  for all $u\in \domh$.
\end{lemma}

\begin{proof}
  Integration by parts shows that if $v \in \domh$, then
  \begin{align*}
&  \ang{\Ph v, v}   = \\
 & \hspace{.5cm} \int \left( \abs{h D_{x}v}^{2} + \abs{\frac{1}{x}\grad_{z}v}^{2} -
      \abs{(h x D_{x} + \lambda)v}^{2} +
      \frac{n^{2}-1}{4}h^{2}\abs{v}^{2}\right)  2i (\Im \lambda)\ang{(h x
        D_{x} + \lambda)v,\overline{v}} ,
  \end{align*}
  where the pairing on the left side is of $\domh$ with $\domh'$.

  We apply this identity to $v = Au \in \domh$ and then first estimate
  \begin{equation*}
    \ang{A\Ph , Au} + \ang{[ \Ph, A]u, u} - 2 i (\Im
    \lambda)\ang{(hxD_{x}+ \lambda)Au, Au}.
  \end{equation*}
  The first term is estimated by Cauchy--Schwarz:
  \begin{equation*}
    \abs{\ang{ A\Ph u, Au}} \leq \frac{1}{4\epsilon}\norm{A\Ph
      u}^{2}_{\domh'} + \epsilon \norm{Au}^{2}_{\domh}.
  \end{equation*}
  Microlocal elliptic regularity lets us estimate $A\Ph u$ in terms of
  $G\Ph u$.  As $\Im \lambda = \bo (h)$, the final term is bounded by
  \begin{equation*}
    Ch \left( \norm{hxD_{x}Au}^{2} + \norm{Au}^{2}\right).
  \end{equation*}
  The additional factor of $h$ allows these terms to be absorbed into
  the $h\norm{Gu}_{\domh}^{2}$ term.

  We now turn to the term involving $[P, A]$.  After applying
  Lemma~\ref{lem:commuting-through-scl} and keeping track of the
  factors lying in $\frac{1}{x}\Diffbh^{1}$ but not $\Diffbh^{1}$, we
  can estimate this term by $h \norm{Gu}^{2}_{\domh}$.
\end{proof}

As we have assumed that the operators in
Lemma~\ref{lemma:integration-by-parts-h} are supported in $\{ x <
c_{0}\}$, we obtain the following corollary, which we record for use
in the hyperbolic section below:
\begin{corollary}
  \label{cor:another-elliptic-reg}
  If $A$ and $G$ are as in Lemma~\ref{lemma:integration-by-parts-h},
  there are constants $C_{0}$ (independent of $A$) and $C$ so that
  \begin{equation*}
    \norm{Au}_{\domh} \leq C_{0} \norm{Au}_{L^{2}} + C\left(
      \norm{G\Ph u}_{\domh'} +  \norm{Gu}_{\domh}\right) + \bo (h^{\infty})\norm{u}_{\domh}.
  \end{equation*}
\end{corollary}

\begin{proof}
  As $x < c_{0}$, we can bound $-\abs{(hxD_{x}+\lambda)Au}^{2}$ below
  by
  \begin{equation*}
    -2c_{0}^{2} \abs{hD_{x}Au}^{2} - 2\abs{\lambda}^{2}\abs{Au}^{2}.
  \end{equation*}
  The first of these terms can be absorbed (together with $\epsilon
  \norm{Au}_{\domh}^{2}$) into the first term on the left in
  Lemma~\ref{lemma:integration-by-parts-h}, while the second term is
  moved to the right side.
\end{proof}

Proposition~\ref{prop:main-ell-est-semicl} then follows immediately
by applying the following lemma and bounding $\norm{G\Ph
  u}_{\domh'}\leq \norm{G\Ph u}_{L^{2}}$:
\begin{lemma}
  \label{lem:elliptic-ingredient}
  Suppose $A$ and $G$ are as in the statement of
  Proposition~\ref{prop:main-ell-est-semicl}.  If, in addition, $A$ is
  supported in $\{ x < \delta / \sqrt{2}\}$ and $\{(\xib + \lambda)^{2} <
  \frac{1}{2}\delta^{-2}(\xib^{2} + \abs{\zetab}^{2})\}$, then
  \begin{equation*}
    \norm{Au}_{\domh} \leq C \norm{G\Ph u}_{\domh'} +
    Ch^{1/2}\norm{Gu}_{\domh} + \bo (h^{\infty})\norm{u}_{\domh}.
  \end{equation*}
\end{lemma}

\begin{proof}
  As $A$ is supported in $\{ x < \delta / \sqrt 2\}$, we know
  \begin{align*}
 &    \delta^{-2}\int \left( \abs{h x D_{x}Au}^{2} + \abs{h
        \grad_{z}u}^{2}\right) - \int \abs{(hxD_{x}+\lambda)Au}^{2}
    \leq \\
    & \hspace{1.5cm} \frac{1}{2}\int \left( \abs{h D_{x}Au}^{2} +
      \abs{\frac{1}{x}\grad_{z}Au}^{2}- \abs{(hxD_{x}+\lambda)Au}^{2}\right).
  \end{align*}
  Our other hypothesis on the support of $A$ shows that we can find operators
  $B,F\in \Psibh^{1}$ with $\WFbh'(A) \subset \ellbh(B)$ so  that
  \begin{equation*}
    Z = \delta ^{-2} \left( (hxD_{x})^{*}(hxD_{x}) +
      \frac{h^{2}}{x^{2}}\lap_{\met}  - \left( h xD_{x} +
        \lambda\right)^{*}\left( hxD_{x} + \lambda\right)\right) - (B^{*}B + hF)
  \end{equation*}
  satisfies
  \begin{equation*}
    \WFbh(Z u) \cap \WFbh'(A) = \emptyset.
  \end{equation*}
  Integrating by parts and applying
  Lemma~\ref{lemma:integration-by-parts-h} shows that
  \begin{align*}
    \norm{BAu}^{2}_{L^{2}} &+ \frac{1}{2}\int \left( \abs{hD_{x}Au}^{2}
    + \abs{\frac{h}{x}\grad_{z}Au}^{2}\right) \leq \epsilon
                             \norm{Au}_{\domh}^{2}+
                             \frac{C}{\epsilon}\norm{G\Ph
                             u}_{\domh'}^{2} \\
    &\quad + Ch \norm{Gu}_{\domh}^{2} + Ch\norm{FAu}\norm{Au} + \bo (h^{\infty})\norm{u}_{\domh}^{2}.
  \end{align*}
  As $B$ is elliptic on $\WFb'(A)$, the left side controls
  $\norm{Au}_{\domh'}^{2}$, while the right side is controlled by
  \begin{equation*}
    \epsilon \norm{Au}_{\domh}^{2} + C \norm{G\Ph u}_{\domh'}^{2} + Ch
    \norm{Gu}_{\domh}^{2} + \bo (h^{\infty})\norm{u}_{\domh}^{2}.
  \end{equation*}
  Absorbing the first term into the left side finishes the proof.
\end{proof}

\subsubsection{Hyperbolic propagation near the singularity}
\label{sec:hyperb-prop}

In this subsection we complete the proof of the third and fourth parts
of Proposition~\ref{prop:mapping-props}.  In particular, we establish
the following proposition:
\begin{proposition}[{cf.~\cite[Proposition 5.8]{GW}}]
  \label{prop:hyperbolic-prop-semicl}
  If $G \in \Psibhc$ is elliptic at $\{(0,z,0,0) \mid z \in Z\}$, then
  there are $\cQ, \cQ_{1} \in \Psibhc$ with $\cQ$ elliptic at $\{ (0,
  z, 0, 0) \mid z \in Z\}$ and
  \begin{align*}
    \WFbh'\cQ &\subset \ellbh(G), \\
    \WFbh'\cQ_{1} &\subset \ellbh (G) \cap \{ - \xib > 0\},
  \end{align*}
  so that for all $u \in \domh$,
  \begin{equation*}
    \norm{\cQ u}_{\domh} \leq \frac{C}{h}\norm{G\Ph u}_{\domh'} + C
    \norm{\cQ_{1}u}_{\domh} + Ch \norm{Gu}_{\domh} + \bo (h^{\infty})\norm{u}_{\domh}.
  \end{equation*}
\end{proposition}
We note that the estimate in
Proposition~\ref{prop:hyperbolic-prop-semicl} immediately implies the estimate
\begin{equation*}
  \norm{\cQ u}_{\cX_{h}^{s_{\tow}}} \leq \frac{C}{h} \norm{G\Ph
    u}_{\cY_{h}^{s_{\tow}-1}} + C\norm{\cQ_{1} u}_{\cX_{h}^{s_{\tow}}}
  + Ch \norm{Gu}_{\cX_{h}^{s_{\tow}}} + \bo(h^{\infty})\norm{u}_{\cX_{h}^{s_{\tow}}},
\end{equation*}
finishing the proof of the fourth part of
Proposition~\ref{prop:mapping-props}.

As in Section~\ref{sec:hyperb-prop-bulk}, we introduce a basic operator
$A\in \Psibhc$ with symbol given by
\begin{equation*}
  a = \chi_{0}(2 - \phi / \delta) \chi_{1}(2 - \xib/\delta)
  \chi_{2}(\xib^{2} + \abs{\zetab}^{2}),
\end{equation*}
where $\chi_{i}$ are the same functions as in that section and $\phi =
-\xib + \frac{1}{\beta^{2}\delta}x^{2}$.  Recall that $\chi_{2}$ is
supported in $[-2c_{1},2c_{1}]$ and identically one on
$[-c_{1},c_{1}]$, so that $a$ is essentially determined by the three
parameters $c_{1}, \beta$, and $\delta$.

We also choose a basic operator $B \in \Psibhc$ with symbol
\begin{equation*}
  b = \frac{2}{\sqrt{\delta}}(\chi_{0}\chi_{0}')^{1/2}\chi_{1}\chi_{2},
\end{equation*}
so that factors of $B$ arise when derivatives land on $\chi_{0}$ in
$A$.

As in that section (and Melrose--Vasy--Wunsch~\cite{MVW} or
Gannot--Wunsch~\cite{GW}), the symbol $a$ is well-localized:
\begin{lemma}
  \label{lemma:semicl-local-symb}
  Given any neighborhood $U$ of $\{(0,z,0,0) \mid z \in Z\}$ in
  $\Tbstar \mf$ and any $\beta > 0$, there are $\delta_{0}>0$ and
  $c_{1}>0$ so that $a$ is supported in $U$ for all $0 < \delta < \delta_{0}$.
\end{lemma}

We now compute the commutator of $\Ph$ with $A^{*}A$:
\begin{lemma}
  \label{lem:semicl-compute-comm}
  With $Q_{0}, Q_{1}$, and $Q_{j}$ denoting the conic vector fields as
  in Section~\ref{sec:semicl-bl-calc}, the commutator of $\Ph$ and
  $A^{*}A$ is given by
  \begin{align*}
   & \frac{i}{h}\left( \Ph^{*}A^{*}A - A^{*}A\Ph\right)= \\ 
    & \hspace{.5cm} -B_{0}\Ph + B^{*}\left(
      C^{*}C + R_{0} + \sum_{j}R_{j}Q_{j} +
      \sum_{j,k}Q_{j}^{*}R_{jk}Q_{k}\right)B + E' + E'' + hR',
  \end{align*}
  where the terms enjoy the following properties:
  \begin{itemize}
  \item $\ds C = hxD_{x} + \lambda$,
  \item $\ds \sigmabh (B_{0}) = 2 \pd[\xib](a^{2})$,
  \item $\ds R_{0}, R_{j}, R_{jk}\in \Psibhc$ satisfy
    \begin{equation*}
      \abs{\sigmabh (R_{\bullet})} \leq C_{1} (\delta \beta + \beta^{-1}),
    \end{equation*}
  \item $\ds R'' \in x^{-2}\Diffbh^{2}\Psibhc$,
  \item $E', E'' \in x^{-2}\Diffbh^{2}\Psibhc$ satisfy
    \begin{equation*}
      \WFbh'(E')\subset \{-\xib >0\}, \quad \WFbh'(E'')\cap \Sigmah = \emptyset.
    \end{equation*}
  \end{itemize}
\end{lemma}

\begin{proof}
  We use Lemma~\ref{lem:commuting-through-scl} to commute $A^{*}A$
  through $\Ph$, using that $A$ is basic.  The main term arising from
  the commutator reproduces the main terms in $\Ph$; indeed, it is of
  the form
  \begin{equation*}
    B^{*}\left( (hD_{x})^{*}(hD_{x}) + \frac{h^{2}}{x^{2}}\lap_{\met}\right)B.
  \end{equation*}
  We use the form of the operator to exchange this term for $B_{0}\Ph$
  and $B^{*}C^{*}CB$.  The other terms in the expression arise in a
  similar way as those in Melrose--Vasy--Wunsch~\cite{MVW} (explained
  above in the proof of Lemma~\ref{lem:commuting-through}).  The term
  arising from $\Ph^{*}-\Ph$ can be absorbed into the $R_{0}$ term as
  the symbol of $A$ is estimated by $\sqrt{\delta}b$.
\end{proof}

We also require that the remainder terms are sufficiently small as to
be estimable:
\begin{lemma}
  \label{lem:estimating-remainders}
  For any $\epsilon > 0$, there are $\beta > 0$ and $\delta_{1} \in
  (0,\delta_{0})$ so that for all $0 < \delta < \delta_{1}$,
  \begin{equation*}
    \abs{\ang{R_{0}Bu, Bu}} + \sum_{j}\abs{\ang{R_{j}Q_{j}Bu, Bu}} +
    \sum_{j,k}\abs{\ang{Q^{*}_{j}R_{jk}Q_{k}Bu, Bu}} \leq \epsilon
    \norm{Bu}^{2}_{\domh} + \bo (h^{\infty})\norm{u}_{\domh}^{2}.
  \end{equation*}
\end{lemma}

\begin{proof}
  As in the proof of Lemma~\ref{lemma:bulk-remainder-terms}, we rely on
  the symbol estimates in Lemma~\ref{lem:semicl-compute-comm}.
  Indeed, we bound
  \begin{align*}
    \norm{R_{\bullet}v}_{L^{2}} &\leq 2 \sup
    \abs{\sigmabh(R_{\bullet})}\norm{v}_{L^{2}} +
                                  \bo(h^{\infty})\norm{v}_{L^{2}}  \\
                                &\leq 2C_{1}\left( \delta \beta +
                                  \beta^{-1}\right)\norm{v}_{L^{2}} +
                                  \bo (h^{\infty})\norm{v}_{L^{2}}.
  \end{align*}
  We now fix $\beta > 0$ sufficiently large and then take
  $\delta_{1}\in (0,\delta_{0})$ sufficiently small to make
  $2C_{1}(\delta_{1}\beta + \beta^{-1}) < \epsilon / 3$.

  We now consider the individual terms.  For the $R_{0}$ term, we
  apply the above inequality with $v=Bu$ and appeal to
  Cauchy--Schwarz.  The $R_{j}$ and $R_{jk}$ terms are nearly
  identical, e.g.,
  \begin{align*}
    \abs{\ang{Q_{j}R_{jk}Q_{k}Bu,Bu}} &= \abs{\ang{R_{jk}Q_{k}Bu,
                                        Q_{j}Bu}}\\
    &\leq 2C_{1}(\delta\beta+\beta^{-1})\norm{Bu}_{\domh}^{2} \leq
      \epsilon\norm{Bu}_{\domh}^{2}. 
  \end{align*}
\end{proof}

We now finish the proof of
Proposition~\ref{prop:hyperbolic-prop-semicl}.
\begin{proof}[Proof of Proposition~\ref{prop:hyperbolic-prop-semicl}]
  Given $u \in \domh$, we apply Lemma~\ref{lem:semicl-compute-comm} to
  write
  \begin{align*}
   &  \frac{2}{h}\Im \ang{A\Ph u, Au} =
                                      \frac{i}{h}\ang{\left(\Ph^{*}A^{*}A - A^{*}A\Ph\right)u,u} \\
    & \hspace{1.5cm} = \norm{CBu}^{2}_{L^{2}} + \ang{R_{0}Bu,Bu} +
      \sum_{j}\ang{R_{j}Q_{j}Bu,Bu} \\
                                    & \hspace{1.75cm}
                                      +\sum_{j,k}\ang{R_{jk}Q_{k}Bu,Q_{j}Bu}
                                      + \ang{E'u,u} + \ang{E''u,u} +
                                      h\ang{R'u,u} - \ang{B_{0}Pu,u}.
  \end{align*}
  As shown above in Lemma~\ref{lem:psib-semiclassical}, $A,B$, and $CB$
  preserve $\domh$, while $B_{0}$ preserves $\domh'$.

  By Corollary~\ref{cor:another-elliptic-reg} and the ellipticity of
  $C$ on $\WFbh'(B)$, there is a constant $c> 0$ so that 
  \begin{equation*}
    c\norm{Bu}_{\domh}^{2} \leq \norm{CBu}^{2}_{L^{2}} +
    C\norm{G\Ph u}_{\domh'}^{2} + Ch \norm{Gu}_{\domh}^{2} +
    \bo(h^{\infty})\norm{u}_{\domh'}^{2},
  \end{equation*}
  where $c>0$ is independent of $\beta$ and $\delta$ and $G$ is
  elliptic on $\WFbh'(B)$.

  We now take $G\in \Psibhc$ to be elliptic on $\WFbh'(B)$ and
  $\cQ_{1}\in \Psibhc$ to be elliptic on $\WFbh'(E')$ with
  $\WFbh'(\cQ_{1})\subset \ellbh(G)\cap\{-\xib > 0\}$.  Applying
  Lemma~\ref{lem:estimating-remainders} yields an estimate of the form
  \begin{align*}
    \frac{c}{2}\norm{Bu}_{\domh}^{2} &\leq \frac{2}{h}\abs{\ang{A\Ph u,
                                       Au}} + C\norm{G\Ph u}_{\domh'}^{2} + Ch\norm{Gu}_{\domh}^{2} \\
    &\quad + \abs{\ang{(E'+E'')u,u}} + h\abs{\ang{R'u,u}} +
      \abs{\ang{B_{0}\Ph u, u}} + \bo (h^{\infty})\norm{u}_{\domh}^{2}.
  \end{align*}
  We estimate the $E'$ term by $\cQ_{1}$ via microlocal elliptic
  regularity and the $E''$ term by
  Proposition~\ref{prop:main-ell-est-semicl}.  The second line is
  therefore bounded by
  \begin{equation*}
    \frac{C}{h}\norm{G\Ph u}_{\domh'}^{2} + Ch \norm{Gu}_{\domh}^{2} +
    C\norm{\cQ_{1}u}_{\domh}^{2} + \bo(h^{\infty})\norm{u}_{\domh}^{2}.
  \end{equation*}

  Because $\WFbh'(A)\subset\ellbh(G)$, we can further estimate
  \begin{equation*}
    \frac{2}{h}\abs{\ang{A\Ph u, Au}} \leq
    \frac{C}{h^{2}\epsilon}\norm{G\Ph u}_{\domh'}^{2} + C\epsilon
    \norm{Au}_{\domh}^{2} + \bo (h^{\infty})\norm{u}_{\domh}^{2}.
  \end{equation*}

  By construction, $\chi_{0}(s) =s^{2}\chi_{0}'(s)$ for $s>0$, and so
  \begin{equation*}
    a = \frac{1}{2}\delta^{1/2}(2-\phi/\delta)b.
  \end{equation*}
  We may therefore write $A = FB + hF'$ for some $F,F'\in \Psibhc$ in
  order to estimate $Au$ by $Bu$.  Putting the above together yields
  the estimate
  \begin{equation*}
    \norm{Bu}_{\domh} \leq \frac{C}{h}\norm{G\Ph u}_{\domh'} +
    C\norm{\cQ_{1}u}_{\domh} + Ch^{1/2}\norm{Gu}_{\domh} +
    \bo(h^{\infty})\norm{u}_{\domh}. 
  \end{equation*}
  Taking $\cQ = B$ finishes the proof.
\end{proof}

\section{Proof of Theorem~\ref{thm:main-thm}}
\label{sec:proof-theor-refth}

This section is devoted to a sketch of the proof of the main theorem,
which is implied by the more refined theorem below:
\begin{theorem}
  \label{thm:main-thm-v2}
  Suppose $w$ is a solution of the wave equation on a cone.  If the
  initial data of $w$ are smooth and compactly supported away from the
  conic singularity, i.e.,
  \begin{align*}
    \Box w &= 0 \text{ on }\reals \times C(Z), \\
    (w, \pd[t]w)|_{t=0} &\in C^{\infty}_{c}(C(Z))\times C^{\infty}_{c}(C(Z)),
  \end{align*}
  then, viewed as a distribution on $[M; S_{+}\cup S_{-}]$,
  \begin{enumerate}
  \item $w$ is conormal to all six boundary hypersurfaces, and 
  \item $w$ is partially polyhomogeneous (i.e., $w \in \pphg^{\cE}([M;
    S_{+}\cup S_{-}])$) at all boundary hypersurfaces other than $\cf$ with index sets
    \begin{equation*}
      \cE =
      \begin{cases}
        \emptyset & \text{at }C_{0}\\
        \left\{ \left(-i\left( \frac{n-1}{2} + j\right), 0\right) \mid j = 0, 1, 2,
          \dots \right\} & \text{at }\scri^{+}, \scri^{-} \\
        \left\{ \left( -i\left( \frac{n}{2} + k + \sqrt{\left(
                \frac{n-2}{2}\right)^{2}+\mu_{j}^{2}}\right), 0\right) \mid
          j, k = 0, 1, 2, \dots \right\} &\text{at }C_{+}, C_{-}
      \end{cases},
    \end{equation*}
    where $\mu_{j}^{2}$ are those eigenvalues of $\lap_{\met}$ on $Z$
    so that
    \begin{equation*}
      \sqrt{\left( \frac{n-2}{2}\right)^2 + \mu_j^2} \notin
      \frac{1}{2} + \integers.
    \end{equation*}
  \end{enumerate}
\end{theorem}

In terms of the radiation field $\mathcal{R}_{+}[w]$, the expansion at
$C_{+}$ implies the expansion in Theorem~\ref{thm:main-thm}.
Theorem~\ref{thm:main-thm-v2} is stronger than
Theorem~\ref{thm:main-thm}, as it implies a joint asymptotic
expansion at $C_{+}\cap \scri^{+}$.  

The proof follows the same outline as in the setting of asymptotically
Minkowski spaces to obtain the existence of the asymptotic expansion;
the key missing steps require extending the propagation and Fredholm
statements near the conic singularities and are formalized in
Propositions~\ref{prop:output-of-bulk-prop}
and~\ref{prop:mapping-props}.  As the same approach works here, we
provide only an abbreviated sketch.

Our strategy is to show first that the solution is partially
polyhomogeneous.  As the initial data are compactly supported, finite
speed of propagation implies that the solution is trivial near
$C_{0}$.  The finite speed of propagation also allows us to replace
$w$ with $\chi w$, where $\chi$ is a smooth cutoff function to a
neighborhood of $\overline{C_{+}}$ in $M$; $\chi w$ is then the
forward solution of an \emph{inhomogeneous} wave equation on
$\reals \times C(Z)$.  We show that $\chi w$ is partially
polyhomogeneous on the blown-up space $[M; S_{+}]$ and an identical
argument near $\overline{C_{-}}$ then establishes the claim for $w$.
Establishing the partial polyhomogeneity of $w$ has as its byproduct
a proof that the index sets at $\scri^{\pm}$ are as stated.
Finally, we establish that the exponents seen in the expansion at
$C_{\pm}$ can be characterized as resonances associated to the
hyperbolic cone with the same link.  It suffices to show this for the
forward solution as the backward solution has the same form near
$\scri^{-}\cup C_{-}$.

We therefore begin by considering the equation
\begin{equation*}
  \Box_{g} w = f'
\end{equation*}
on $M^{\circ}$, where $f' \in \CI_{c}$ and suppose $w$ is the forward
solution.  By translating in time and replacing $w$ with $\chi w$, we
may assume that $f'$ (and therefore $w$) is supported in the forward
light cone $\{t > r\}$ and in $\{ t > 1\}$.  With $\rho$ denoting a
defining function\footnote{Near $S_{+}$, the primary region of
  interest, we recall that $\rho = t^{-1}$.} for $\mf$ and $x$ a
defining function for $\cf$, we consider the conjugated equation
\begin{equation*}
  Lu = f,
\end{equation*}
where
\begin{align*}
  L &\equiv \rho^{-\frac{n-1}{2}-2}\Box_{g}\rho^{\frac{n-1}{2}}, \\
  u &= \rho^{-\frac{n-1}{2}}w \in C^{-\infty}(M), \\
  f &= \rho^{-\frac{n-1}{2}-2}f' \in \CI_{c}(M^{\circ}).
\end{align*}
This conjugation and rescaling transform $\Box_{g}$ into $L$, a
``wedge-$\bl$-differential-operator'', i.e., a $\bl$-differential
operator at $\mf$ and a wedge-type operator at $\cf$.  Note that the
partial polyhomogeneity of $u$ implies that of $w$ with index sets
shifted by $(n-1)/2$.

Due to the scaling invariance (in the variable $\rho$) of the metric,
$L$ agrees with its normal operator, so $N(L) \equiv L$.  This
observation greatly simplifies the analysis of the problem by
eliminating remainder terms and thus allows us to avoid an additional
iterative argument; the lack of remainder terms accounts for the
absence of logarithmic terms in the expansions of
Theorem~\ref{thm:main-thm-v2}.\footnote{If we instead perturb the
  spacetime metric, the remainder terms can be handled as in the
  asymptotically Minkowski setting~\cite{BVW2}.}

For convenience, we recall from
Section~\ref{sec:operators} the form of the operator $L$
in a neighborhood of $\overline{C_{+}}$ in the coordinate system given
by $(\rho = 1/t, x = r/t, z)$.  As we eventually pass to the blown-up
space $[M; S_{+}]$, it is often convenient to include a coordinate
defining $S_{+}$.  We therefore also include the coordinate systems
$(\rho = 1/t, v = (t-r)/t,z)$, which are valid in a neighborhood of $S_{+}$.

Near $S_{+}$, in the coordinate system given by $(\rho = 1/t, x = r/t, z)$, $L$ has the
following form:
\begin{equation*}
  L = \left( \rho D_{\rho} + x D_{x}\right)^{2} - n i \left( \rho
    D_{\rho} + x D_{x}\right) - D_{x}^{2} + \frac{(n-1)i}{x}D_{x} -
  \frac{1}{x^{2}}\lap_{\met} - \frac{n^{2}-1}{4}.
\end{equation*}

Similarly, in terms of $(\rho = 1/t, v = (t - r)/t, z)$, $L$ takes the
form
\begin{equation*}
  L = \left( \rho D_{\rho} - (1-v)D_{v}\right)^{2} - n i \left( \rho
    D_{\rho} - (1-v)D_{v}\right) - D_{v}^{2} - \frac{(n-1)i}{1-v}D_{v}
  - \frac{1}{(1-v)^{2}}\lap_{\met} - \frac{n^{2}-1}{4}.
\end{equation*}

After applying the Mellin transform to the identity $Lu = f$, we
obtain a family of equations\footnote{Recall that $u$ has already been
  localized to $\{ \rho < 1\}$, so it is unnecessary to include an
  additional cut-off function here.}
\begin{equation*}
  \Ps \us = \fs,
\end{equation*}
where $\Ps = \widehat{N}(L)$ is the reduced normal operator of $L$.
As $w$ vanishes near $\overline{C_{-}}$ in $M$, we may arrange that $\us$
also vanishes in a neighborhood of $\overline{C_{-}}$ in $X = \mf$.
In fact, as we are able to assume that $f'$ and $w$ are supported in
the interior of the forward light cone $\{ t > r\}$, we may further
assume that $\fs$ and $\us$ are supported in $\overline{C_{+}}$.

We start by showing that $\us$ lies in the following space of conormal distributions:
\begin{definition}
  Suppose $\ut$ is a distribution on $X = \mf$.  We say that
  $\ut \in I^{(s)} (S_{+})$ if
  \begin{enumerate}
  \item $\ds \ut \in\Hb^{s}(X)$,
  \item away from $S_{+}$, $\ds\ut \in \Hb^{\infty}(X)$, and 
  \item if $V_{1}, \dots, V_{r}$ are $\bl$-vector fields on $X$ with
    principal symbols vanishing on $N^{*}S_{+}$, then
    $V_{1}\dots V_{r}\ut \in \Hb^{s}(X)$.
  \end{enumerate}
\end{definition}
In other words, $\ut \in I^{(s)}(S_{+})$ if it lies in $\Hb^{s}$ and
lies in $\Hb^{\infty}(X)$ away from $S_{+}$.

One consequence of Proposition~\ref{prop:output-of-bulk-prop} and
mapping properties of the Mellin transform is the following proposition:
\begin{proposition}
  \label{prop:conormal-space}
  There are $\varsigma_{0}, s$ so that $\us$ is holomorphic on the
  upper half-plane $\Im\sigma > -\varsigma$ taking values in
  $I^{(s-0)}(S_{+})$ and obeys the following estimate for each $N$ and
  each seminorm $\norm{\bullet}$ on $I^{(s)}(S_{+})$:
  \begin{equation*}
    \sup_{\Im \sigma > -\varsigma_{0}} 
    \int_{\Im \sigma = C}\norm{\us}^{2} \langle \sigma \rangle ^{N}\, d(\Re \sigma) < \infty.
  \end{equation*}
\end{proposition}

In order to aid in bookkeeping, we introduce a compact name for these
spaces.  In what follows, $\mathcal{H}(\Omega)$ refers to the space
of holomorphic functions on the domain $\Omega \subset \complexes$.
\begin{definition}
  For $\varsigma, s, \in \reals$, we let $\complexes_{\varsigma}$
  denote the upper half-plane $\Im \sigma > - \varsigma$ and then
  define
  \begin{equation*}
    \mathcal{B}(\varsigma, s) = \mathcal{H}(\complexes_{\varsigma})
    \cap \langle \sigma\rangle^{-\infty}L^{\infty}_{\Im \sigma}
    L^{2}(\reals_{\Re\sigma} ; I^{(s)}(S_{+})).
  \end{equation*}
  In other words, $\mathcal{B}(\varsigma, s)$ consists of those
  $g_{\sigma}$ holomorphic in $\sigma \in \complexes_{\varsigma}$
  taking values in $I^{(s)}(S_{+})$ so that for each seminorm on
  $I^{(s)}(S_{+})$,
  \begin{equation*}
    \int_{-\infty}^{\infty}\norm{g_{\mu + i \nu}}^{2}_{\bullet}\langle \mu\rangle^{2k}\,d\mu
  \end{equation*}
  is uniformly bounded in $\nu > -\varsigma$.
\end{definition}
Observe that because $f \in
\CI_{c}(M^{\circ})$, we have 
\begin{equation*}
  \fs \in \mathcal{B}(C, s') \text{ for all }C, s'.
\end{equation*}

Proposition~\ref{prop:conormal-space} can be restated as saying that
there are $\varsigma_{0}$, $s$ so that $\us
\in\mathcal{B}(\varsigma_{0}, s-0)$.  We now turn our attention  to
its proof.

\begin{proof}[Proof of Proposition~\ref{prop:conormal-space}]
  Because $\rho^{(n-1)/2}u$ lies in some $\Hbdst^{s,\gamma}(M)$, we
  have
  \begin{equation}
    \label{eq:initial-est-for-ws}
    \us \in \mathcal{H}(\complexes_{\varsigma_{0}}) \cap \langle
    \sigma \rangle ^{\max(0,-s)} L^{\infty}L^{2}(\reals , H^{s}(\mf)),
  \end{equation}
  where $\varsigma_{0} = \gamma - (n-1)/2$.  By reducing $\gamma$, we may
  assume that $s + \gamma < 1/2$ so as to be able to apply the
  regularity results of Proposition~\ref{prop:output-of-bulk-prop}.
  We may also arrange that $\us$ vanishes in a neighborhood of
  $\overline{C_{-}}$ in $\mf$ because $u$ vanishes near
  $\overline{C_{-}}$ in $M$.

  Proposition~\ref{prop:output-of-bulk-prop} implies that $w$ is
  jointly conormal to $S_{+}$ and $\cf$ and so by the mapping
  properties of the Mellin transform (see, e.g., an earlier work in
  this series~\cite[Lemma 2.3]{BVW1}),
  \begin{equation*}
    \us \in \mathcal{B}(\varsigma_{0}, -\infty).
  \end{equation*}
  Interpolating with equation~\eqref{eq:initial-est-for-ws} yields the result.
\end{proof}

Having placed $\us$ in the holomorphic conormal space
$\mathcal{B}(\varsigma_{0}, s-0)$, we may begin the
inversion procedure.  Because
\begin{equation*}
  \Ps\us = \fs,
\end{equation*}
our aim is to invert $\Ps$ and employ a contour-shifting argument to
enlarge the domain of meromorphy for $\us$.

By Proposition~\ref{prop:mapping-props}, $\Ps^{-1}$ forms a
meromorphic family in any strip in the complex plane (though its
domain and range are dependent on the location of the strip).  As
$\fs$ is entire, writing $\us = \Ps^{-1}\fs$, we see that
$\us$ is meromorphic in any upper half-plane taking values in the
$\cX$ spaces.  More precisely, we shift the contour $N$ units to see that
$\us$ is meromorphic in the half plane $\Im \sigma > - \varsigma_{0}-N$
with values in $\cX^{s_{\tow}}$, where $s_{\tow}|_{\Lambda^{+}} <
\frac{1}{2} - \varsigma_{0}-N$.  In particular, $\us$ is meromorphic
with values in $\langle \sigma \rangle^{-\infty}L^{\infty}L^{2}(\reals
; \Hb^{\min (s-0, \frac{1}{2}-\varsigma_{0}-N)})$.  On the other hand,
since $\Ps$ maps the expression to a conormal space, it must in fact
take values in the conormal space
\begin{equation*}
  \langle \sigma \rangle^{-\infty} L^{\infty}L^{2}(\reals ; I^{(\min
    (s-0, \frac{1}{2}-\varsigma_{0}-N-0))}),
\end{equation*}
by propagation of the propagation results of
Section~\ref{sec:propagation-at-conic} as well as the first case of
Theorem~6.3 of Haber--Vasy~\cite{Haber-Vasy}, which concerns the
propagation of Lagrangian regularity into conic Lagrangian
submanifolds of radial points.

We have therefore shown that for any $N$,
\begin{align*}
  \us &\in \mathcal{B}(\varsigma_{0} + N, \min (s-0,
        1/2-\varsigma_{0}-N-0)) + \sum
        _{\substack{(\sigma_{j},m_{j})\in \cE_{0}\\-\varsigma_{0} >
  \Im \sigma_{j} > - \varsigma_{0} - N}} (\sigma - \sigma_{j})^{-m_{j}}a_{j},
\end{align*}
where $\cE_{0}$ is the set of poles of $P_{\sigma}^{-1}$ and
\begin{equation*}
  a_{j} \in \mathcal{B}(\varsigma_{0}+N, \Im \sigma_{j} + 1/2-0).
\end{equation*}

After inverting the Mellin transform, we conclude that $u$ enjoys a
partial asymptotic expansion.  In fact, on $M$, we have
\begin{equation*}
  u = \sum_{\substack{(\sigma_{j}, k)\in \cE_{0} \\ -\varsigma_{0} >
      \Im \sigma_{j} > -\ell}} \rho^{i\sigma_{j}}(\log \rho)^{k}
  b_{jk} + u',,
\end{equation*}
where, for some $C = s + \varsigma_{0}$ (with $s$ as in Proposition~\ref{prop:conormal-space}),
\begin{equation*}
  u' \in \rho^{\ell}\Hb^{\min (C-\ell-0, 1/2-\varsigma_{0}-\ell-0)}(M).
\end{equation*}
The coefficients $b_{jk}$ are smooth functions of $\rho$ taking values
in $I^{(1/2+\Im \sigma_{j}-0)}$.  Looking further into the asymptotic
expansion of $u$, one finds that the coefficients and the remainder
term are growing more singular owing to the radiation field ``hiding''
at $S_{+}$.

In fact, after blowing up $S_{+}$,
Proposition~\ref{prop:2-step-phg} implies that the same
arguments in the preceding discussion provide one step toward the
joint partial polyhomogeneity of $u$.  Indeed, $u$ enjoys an asymptotic
expansion at $C_{+}$ uniformly up to the corner $\mf \cap \scri^{+}$
in $[M; S_{+}]$.

The other needed step involves estimates at
$\scri^{+}$.  This argument relies on the observation that on $M$, the
operators $L$ and $2D_{v}\left( \rho D_{\rho} + v D_{v}\right)$ differ
only by terms with additional vanishing at $N^{*}S_{+}$.  The vector
field $\rho D_{\rho} + v D_{v}$ lifts to the $\bl$-normal vector field
for $\scri^{+}$ in $[M ; S_{+}]$.  Writing
\begin{equation*}
  R = \rho D_{\rho} + v D_{v}, 
\end{equation*}
the other step establishing the polyhomogeneity of $u$ requires that
$u$ enjoy additional vanishing after the application of $(R+ik)\dots
(R+i)R$.

We ignore for now the additional terms in $L$\footnote{Of course,
  these additional terms are always there.  Managing these terms forms
  a sizable part of Section~9.2 of the previous paper~\cite{BVW2}
  and we refer the reader there for a thorough discussion.} and
suppose $L = 2 D_{v}(\rho D_{\rho} + v D_{v}) = 4D_{v}R$.  As these
statements are local to $S_{+}$, a simple argument with cut-off
functions shows we are free to ignore the differentiation near the
conic singularity ($x=0$).

As $Lu$ is smooth and compactly supported, and $u \in \Hb^{s,\gamma}$,
we know that $D_{v}Ru \in \Hb^{s,\gamma}$.  Because $D_{v}$ is
elliptic on $\WFb(u)$, it is microlocally invertible and so $Ru \in
\Hb^{s+1,\gamma}$, i.e., $Ru$ is one order better than $u$.

To continue this iterative process, observe that $RD_{v} =
D_{v}(R+i)$, so that
\begin{equation*}
  \left( \prod _{j=0}^{k-1}(R+ij)\right) L = \left( \prod
    _{j=0}^{k-1}(R+ij)\right) D_{v}R = D_{v}\left( \prod_{j=0}^{k}(R+ij)\right).
\end{equation*}
An inductive argument then shows that
\begin{equation*}
  \left( \prod_{j=0}^{k}(R+ij)\right) u \in \Hb^{s+k+1,\gamma}(M),
\end{equation*}
so that $(R+ik)\dots (R+i)Ru$ enjoys $k+1$ additional orders of
regularity at $S_{+}$.

As $u$ is already conormal to $S_{+}$, measure of regularity there are
essentially based on applications of $D_{v}$.  The vector field
$vD_{v}$ is tangent to $S_{+}$ (and so can be applied to $u$ as many
times as we like), so we may interpret additional regularity at
$S_{+}$ as additional vanishing at $S_{+}$.\footnote{This
  interpretation can be formalized by an integration argument and
  requires keeping track of the factors of the module for which $w$
  already possesses regularity; see~\cite[Section~9.2]{BVW2} for
  details.}  This extra vanishing is precisely what is needed for the
application of Proposition~\ref{prop:2-step-phg} and completes
the bulk of the proof of Theorem~\ref{thm:main-thm}.

We finally characterize the exponents seen in
Theorems~\ref{thm:main-thm} and~\ref{thm:main-thm-v2}.  As noted
above, these exponents are the poles of $\Ps^{-1}$ acting as an
operator $\cX^{s_{\tow}}\to \cY^{s_{\away}}$.

As $\Ps$ is the Mellin conjugate of $L$, we may write
\begin{equation*}
  \Ps = - (1-x^{2})D_{x}^{2} - i(n+1+2i\sigma)x D_{x} +
  \frac{n-1}{x}iD_{x} - \frac{1}{x^{2}}\lap_{\met} + \sigma^{2} - n i
  \sigma + \frac{n^{2}-1}{4}.
\end{equation*}
In particular, in $C_{+}$ we have the following identity:
\begin{align*}
&   \left(1-x^{2}\right)^{\frac{n-1}{4}+i\frac{\sigma}{2}+1}\Ps
  \left(1-x^{2}\right)^{-\frac{n-1}{4}-i\frac{\sigma}{2}} = \\
  & \hspace{1cm} -\left(
    (1-x^{2})D_{x}\right)^{2} + i\frac{n-1}{x}(1-x^{2})D_{x} -
  \frac{1-x^{2}}{x^{2}}\lap_{\met} + \left( \frac{n-1}{2}\right)^{2} + \sigma^{2}.
\end{align*}
Taking $x = \tanh r$ identifies $C_{+}$ with the hyperbolic cone
$C_{\mathrm{hyp}}(Z)$ over $(Z,h)$; the conjugation above yields
\begin{equation*}
    \left(1-x^{2}\right)^{\frac{n-1}{4}+i\frac{\sigma}{2}+1}\Ps
  \left(1-x^{2}\right)^{-\frac{n-1}{4}-i\frac{\sigma}{2}} = - \left(
    \lap_{C_{\mathrm{hyp}}(Z)} - \left( \frac{n-1}{2}\right)^{2}-\sigma^{2}\right).
\end{equation*}

Using this identification, for $\fs$ compactly supported in $C_{+}$, a
straightforward adaptation of the arguments\footnote{In fact, this
  adaptation is not necessary; in this structured setting the operator
  $\Ps^{-1}$ can be found explicitly in terms of hypergeometric
  functions, though we do not include it here.} in previous
work~\cite[Section~7]{BVW1} shows that
\begin{equation*}
  \Ps^{-1}f |_{C_{+}} = -\left(
    1-x^{2}\right)^{-\frac{n-1}{4}-i\frac{\sigma}{2}}R_{C_{\mathrm{hyp}}(Z)}(\sigma) \left((1-x^{2})^{1+\frac{n-1}{4}+i\frac{\sigma}{2}}f\right).
\end{equation*}
Here $R_{C_{\mathrm{hyp}}(Z)}(\sigma) = (\lap_{C_{\mathrm{hyp}}(Z)} -
\frac{(n-1)^{2}}{4} - \sigma^{2})^{-1}$ is the resolvent of the
Laplacian on the hyperbolic cone that is invertible for $\Im \sigma \gg
0$.  The exponents appearing in the expansion of $u$ are therefore the
poles of the resolvent on the hyperbolic cone; these poles were found
explicitly in a previous paper of the authors~\cite{resonances}.


\def\cprime{$'$} \def\cftil#1{\ifmmode\setbox7\hbox{$\accent"5E#1$}\else
  \setbox7\hbox{\accent"5E#1}\penalty 10000\relax\fi\raise 1\ht7
  \hbox{\lower1.15ex\hbox to 1\wd7{\hss\accent"7E\hss}}\penalty 10000
  \hskip-1\wd7\penalty 10000\box7}

\end{document}